\documentclass[svgnames]{amsart}
\usepackage[english]{babel}
\usepackage[utf8]{inputenc}
\usepackage[margin=3cm]{geometry}

\usepackage{svg}
\usepackage{hyperref}
\hypersetup{colorlinks,citecolor=Black,linkcolor=Black,urlcolor=Blue}
\usepackage[nameinlink]{cleveref}

\usepackage{thmtools}

\usepackage{tabularx, ifthen}
\usepackage{amssymb, amsfonts, amsmath, amsthm, amscd, amsxtra, latexsym, mathabx}
\usepackage{epigraph}
\usepackage{enumitem}
\usepackage{tikz}\usetikzlibrary{cd}

\theoremstyle{plain}
\newtheorem{thm}{Theorem}[section] 
\newtheorem{lem}[thm]{Lemma} 
\AddToHook{env/lem/begin}{\crefalias{thm}{lem}}
\newtheorem{claim}[thm]{Claim}
\AddToHook{env/claim/begin}{\crefalias{thm}{claim}}
\newenvironment{claimproof}{\begin{proof}}{\end{proof}}
\newtheorem{prop}[thm]{Proposition}
\AddToHook{env/prop/begin}{\crefalias{thm}{prop}}
\newtheorem{cor}[thm]{Corollary}
\AddToHook{env/cor/begin}{\crefalias{thm}{cor}}

\AddToHook{env/corintro/begin}{\crefalias{thmintro}{corintro}}

\AddToHook{env/questintro/begin}{\crefalias{thmintro}{questintro}}

\theoremstyle{definition}
\newtheorem{defn}[thm]{Definition}
\AddToHook{env/defn/begin}{\crefalias{thm}{defn}}
\newtheorem{assumption}[thm]{Assumption}
\AddToHook{env/assumption/begin}{\crefalias{thm}{assumption}}
\newtheorem{rem}[thm]{Remark}
\AddToHook{env/rem/begin}{\crefalias{thm}{rem}}
\newtheorem{ex}[thm]{Example}
\AddToHook{env/ex/begin}{\crefalias{thm}{ex}}
\newtheorem{construction}[thm]{Construction}
\AddToHook{env/construction/begin}{\crefalias{thm}{construction}}

\AddToHook{env/remintro/begin}{\crefalias{thmintro}{remintro}}

\newcommand{\Z}{\mathbb{Z}}
\newcommand{\R}{\mathbb{R}}
\newcommand{\Cay}[2]{\operatorname{Cay}\left(#1,#2\right)}

\newcommand{\frakS}{\mathfrak S}
\newcommand{\diam}{\mathrm{diam}}
\newcommand{\C}{\mathcal C}

\newcommand{\dist}{\mathrm{d}}
\newcommand{\nest}{\sqsubseteq}
\newcommand{\propnest}{\sqsubsetneq}
\newcommand{\orth}{\bot}
\newcommand{\trans}{\pitchfork}

\newcommand{\Stab}[2]{\operatorname{Stab}_{#1}(#2)}

\newcommand{\N}{\mathcal{N}}

\newcommand{\gate}{\mathfrak{g}}

\newcommand{\link}[2]{\operatorname{Link}_{#1}\left(#2\right)}
\newcommand{\Cstar}[2]{\operatorname{Star}_{#1}\left(#2\right)}

\newcommand{\MCG}{\mathcal{MCG}^\pm}
\newcommand{\Aut}[1]{\operatorname{Aut}\left(#1\right)}

\newcommand{\tsh}[1]{\left\{\kern-.7ex\left\{#1\right\}\kern-.7ex\right\}}

\title[A continuum of coarse medians for MCG and RA*G]{A continuum of coarse medians for all mapping class groups and right-angled groups}

\author[G. Mangioni]{Giorgio Mangioni}
    \address{(Giorgio Mangioni) Maxwell Institute and Department of Mathematics\\ Heriot-Watt University\\ Edinburgh, UK\\ Orcid: 0000-0003-2868-5032}
    \email{gm2070@hw.ac.uk}

\begin{document}

\begin{abstract} It is shown that the mapping class group of any non-sporadic surface admits a continuum of equivariant coarse median structures, thus improving a previous theorem of the author. The same holds for a wide class of right-angled Artin and Coxeter groups, including examples of any dimension at least two, and RACGs which are not quasi-isometric to RAAGs; the coarse median structures in these cases are not induced by cocompact cubulations. The results are instances of a general procedure, which uses quasicocycles to flexibly modify the structure of a (suitable) hierarchically hyperbolic group.
\end{abstract}

\maketitle

\epigraph{Auream quisquis mediocritatem\\
diligit, tutus caret obsoleti\\
sordibus tecti, caret invidenda\\
sobrius aula.}
{Quintus Horatius Flaccus, \emph{Carmina}, II 10}

\section*{Introduction}

In order to capture the large-scale ``cubical'' geometry of mapping class groups and hyperbolic groups, in \cite{Bowditch_CM} Bowditch introduced the notion of a \emph{coarse median} on a group $G$, a ternary operation which, roughly, assigns ``coarse centres'' to geodesic triangles in a Cayley graph for $G$ and behaves like the combinatorial median in a CAT(0) cube complex. Two coarse median on the same group are equivalent if they lie within bounded distance from each other, and a \emph{coarse median structure} is an equivalence class of coarse medians. There is also a natural notion of equivariance for coarse median structures on a group, which roughly amounts to the fact that the coarse median is preserved by ``push-forward" under the elements of the group. See \Cref{sec:CM} for details.

It is natural to ask if a given group admits a unique (equivariant) coarse median structure. The question is further motivated by the fact that coarse median preserving automorphisms and their fixed point sets often enjoy better properties than general automorphisms \cite{Elia:coarse_med_preserving}, and all automorphisms are clearly coarse median preserving in a group with a unique coarse median structure. This paper proves that two families of groups which are ubiquitous in geometric group theory, namely mapping class groups and (several) right-angled Artin and Coxeter groups (RA*Gs for short), can be endowed with a plethora of coarse median structures. 

Let us start with mapping class groups. The following proves \cite[Conjecture G]{short_HHG:I} and positively answers a question of Wright (personal communications):
\begin{restatable}{thmintro}{mcg}\label{thmintro:mcg}
Let $S$ be an orientable surface of finite-type, of complexity at least $2$.  The extended mapping class group of $S$, where one allows orientation-reversing homeomorphisms, admits a continuum of equivariant coarse median structures.
\end{restatable}
\noindent The restriction on the complexity is necessary. Indeed, the mapping class group of a surface of complexity at most $1$ is hyperbolic, and therefore admits a unique equivariant coarse median structure by a result of Niblo, Wright, and Zhang \cite[Theorem 4.2]{NWZ} (see also \cite[Corollary 3.2]{FS_pantsgraph} for a self-contained proof). The case of the five-punctured sphere was already settled in \cite{short_HHG:I}; as we shall see later, this paper expands the techniques used there.

Our second main theorem concerns right-angled Coxeter groups. Recall that a RACG $W_\Gamma$ acts properly and cocompactly on its Davis complex $\Sigma_\Gamma$, which is a CAT(0) cube complex. As such, $W_\Gamma$ admits a \emph{cubical coarse median}, induced by pulling back the cubical median operation on $\Sigma_\Gamma$ via an equivariant quasi-isometry $W_\Gamma\to \Sigma_\Gamma$. This coarse median structure is \emph{canonical}, as it is invariant under all automorphisms by work of Fioravanti \cite[Proposition A]{Elia:coarse_med_preserving}. Furthermore, Fioravanti-Levcovitz-Sageev \cite[Theorem B]{CoarseCubicalRig} proved that any two cocompact cubulations on a RACG induce the same coarse median structure, under suitable conditions which are satisfied by ``generic" RACGs by work of Behrstock, \c{C}i\c{c}eksiz, and Falgas-Ravry \cite{BehrstockCiceksizFalgasRavry:connectivity}, and conjectured that any RACG admits finitely many cubical coarse median structure. Fioravanti even asked if RACGs admit finitely many coarse median structures, even beyond the cubical setting \cite[Question 4]{Elia:coarse_med_preserving}. This is not the case for a wide family of RACGs:

\begin{restatable}{thmintro}{coxeter}\label{thmintro_coxeter} 
Let $\Gamma$ be a finite simplicial graph. Suppose that there exist two vertices $v_-, v_+$ such that $\link{\Gamma}{v_+}=\link{\Gamma}{v_-}$ has at least three vertices, is not a join of two proper subgraphs, and is not contained in the link of any other vertex $w\neq v_+,v_-$. Then the right-angled Coxeter group $W_{\Gamma}$ admits a continuum of equivariant coarse median structures.
\end{restatable}
\noindent Since there are at most countably many cocompact cubulations, a cardinality argument implies that uncountably many of the new coarse medians are not cubical. Actually, \emph{none} of the new coarse median structures is cubical: indeed they all restrict to the standard one, induced by the Davis complex, on product of Dihedral parabolic subgroups (see \Cref{cor:new_median_is_the_same_on_DT_RACG}), and the only cubical coarse median with this property is the standard one, as a consequence of \cite[Theorem B(2)]{CoarseCubicalRig} (see e.g. the proof of \cite[Corollary C]{CoarseCubicalRig}).

The same result applies to right-angled Artin groups satisfying analogue combinatorial requirements, thus significantly contributing to the classification of which RAAGs admit unique coarse median structures \cite[Problem at page 2]{FS_pantsgraph}:

\begin{restatable}{thmintro}{artin}\label{thmintro_artin} 
Let $\Gamma$ be a finite simplicial graph. Suppose that there exists a vertex $v\in V(\Gamma)$ such that $\link{\Gamma}{v}$ has at least two vertices, is not a join of two proper subgraphs, and is not contained in the link of any other vertex $w\neq v$. Then the right-angled Artin group $A_{\Gamma}$ admits a continuum of equivariant coarse median structures. 
\end{restatable}

\noindent By \Cref{cor:new_median_is_the_same_on_DT_RAAG} below, all coarse medians above agree with the standard one, coming from the action on the universal cover of the Salvetti complex, on Abelian parabolic subgroups. In turn, \cite[Lemma 3.24]{CoarseCubicalRig} implies that none of the new coarse median structures is cubical.

In both theorems above, the conditions on $\Gamma$ are ``local'', in that they only involve vertices at distance at most two from $v_+$ and $v_-$ (resp. $v$)  in the path metric of $\Gamma$. In particular, \Cref{thmintro_coxeter} applies to RACGs of any dimension, such as the ones whose defining graphs are as in \Cref{fig:coxeter}. 

\begin{figure}[htp]
\centering
\includegraphics{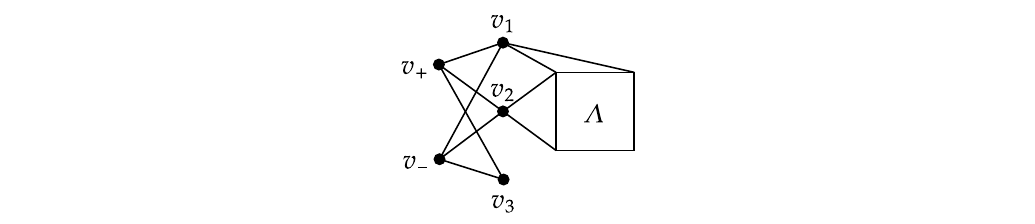}
\caption[alt=''A graph satisfying the requirements of Theorem B, which can contain any graph as an induced subgraph.'']{An example of a graph $\Gamma$ satisfying the requirements of \Cref{thmintro_coxeter}. Here $\Lambda$ is any finite simplicial graph, and some of its vertices are connected to $v_1$ and $v_2$. By suitably choosing $\Lambda$, one can also ensure that $W_\Gamma$ is one-ended.}
\label{fig:coxeter}
\end{figure}

In \cite{short_HHG:I}, I produced a continuum of equivariant coarse median structures for RAAGs on connected graphs without triangles and squares. Since every RAAG is commensurable to a RACG by a theorem of Davis-Januszkiewicz \cite{raag_into_racg}, this implied that several RACGs have uncountably many \emph{non-equivariant} coarse median structures, induced by pulling back the structures on RAAGs through the (not necessarily equivariant) quasi-isometries. Our Theorems~\ref{thmintro_coxeter} and~\ref{thmintro_artin} improve these results in two directions. Firstly, the coarse median structures for RACGs I build here are genuinely equivariant. Furthermore, the class of RA*Gs consider in this paper is significantly larger, as it includes examples of any dimension and RACGs which are not even quasi-isometric to RAAGs (for instance, one can choose $\Lambda$ in \Cref{fig:coxeter} in such a way that $\Gamma$ is not \emph{CFS}, in the sense of \cite{Dani_Thomas}, and in particular $W_\Gamma$ has at least cubical divergence by \cite{Dani_Thomas, BHS_Coxeter, Levcovitz}; on the other hand, one-ended RAAGs have at most quadratic divergence by \cite{Behrstock_Charney}).

\subsection*{Methods and tools} The above theorems are all instances of a general result, \Cref{thm:uncountably_many_CM}. Very roughly, given a finitely generated group $G$ satisfying \Cref{notation:tweakable_hhg}, we build a collection of Lipschitz projections from $G$ to certain \emph{quasilines}, i.e. metric spaces which are quasi-isometric to the real line; together with additional data, these are used to construct a coarse median structure on $G$ making the projections coarse median preserving (with respect to the natural median structure on $\R$). The quasilines themselves are constructed from suitable \emph{quasicocycles} on certain subgroups of $G$, which arise from actions on hyperbolic spaces. Changing the quasicocycles will then yield different coarse median structures. Let us expand on the various parts of the process.

\subsubsection*{Hierarchical hyperbolicity} We work in the framework of \emph{hierarchically hyperbolic groups} (HHG), first introduced by Behrstock-Hagen-Sisto in \cite{HHS_I}. Very roughly, a HHG is equipped with a family of coarsely Lipschitz \emph{coordinate projections} to Gromov-hyperbolic spaces $\pi\colon G\to \C U$, one for each domain $U$ in some index set $\frakS$: think of $\Z^2$ as a lattice in $\R^2$, with the projections to the coordinate axes. The domains are arranged in a partial order (indeed, a hierarchy), so that lower domains are ``peripheral" in higher domains, in analogy with the situation for relatively hyperbolic groups. For mapping class groups, domains are subsurfaces, the partial order is inclusion, and coordinate spaces are \emph{curve graphs}, as first introduced by Harvey \cite{Harvey}; the minimal layer of the structure corresponds to annuli, whose \emph{annular curve graphs} are quasilines with a slightly different definition (see Masur-Minsky \cite[Section 2.4]{MM2}).

As explained in \Cref{ex:coarse_median_from_HHS}, a HHG $(G,\frakS)$ admits an equivariant coarse median structure $\mu_\frakS$. The latter is characterised as the unique coarse median structure making all coordinate projections $\pi_U\colon G\to \C U$ coarse median preserving, with respect to the unique coarse median structure on each hyperbolic spaces $\C U$ (which roughly amounts to taking centroids of triangles). Changing the HHG structure might yield different coarse median structures, so the core of this paper is to exhibit a family of HHGs, including both mapping class groups and the RACGs from \Cref{thmintro_coxeter}, whose hierarchical structure can be flexibly modified.

\subsubsection*{Tweaking the hierarchy} 
More specifically, we start with a HHG $(G,\frakS)$ as in \Cref{notation:tweakable_hhg}, whose bottom layer of coordinate spaces consists of quasilines. In \Cref{prop:new_hhg_structures}, which is the technical heart of the paper, we then show how to replace the quasilines with suitable new ones, while keeping all other layers the same. The resulting HHG structures then produce different coarse median structure: this is proven in \Cref{thm:uncountably_many_CM}, following much of the same arguments as in \cite[Theorem 7.1]{short_HHG:I}. 

Instead of proving the most general result for HHGs, I preferred to tailor  \Cref{notation:tweakable_hhg} so that both mapping class groups and RA*Gs would satisfy them. Nonetheless, I expect that similar techniques, and possibly even \Cref{thm:uncountably_many_CM} as stated, should apply to several other HHGs of interest. 

\subsubsection*{Quasilines from quasicocycles}
The quasilines we feed into \Cref{prop:new_hhg_structures} are built as follows. We fix a minimal domain $\ell$, and consider the homomorphism $\sigma\colon \Stab{G}{\ell}\to \{\pm 1\}$ induced by the action on $\partial \C\ell=\{\pm \infty\}$. By results of Abbott, Balasubramanya, and Osin \cite{ABO}, later generalised by Balasubramanya, Fournier-Facio, and Genevois \cite{NL}, one can construct a quasiline $L$ with a $\Stab{G}{\ell}$-action from any non-trivial \emph{$\sigma$-quasicocycle}, i.e. a map $\theta\colon \Stab{G}{\ell}\to \R$ satisfying the cocycle relation (with respect to $\sigma$) up to a uniformly bounded error. The quasiline is built in a way that an element $g\in \Stab{G}{\ell}$ acts loxodromically precisely when $\theta$ is unbounded on $\langle g\rangle$.

Using the mapping class group as an example, to ensure that $L$ can replace the annular curve graph $\C \ell$, we roughly need that the Dehn twist $T$ around the annulus $\ell$ acts loxodromically on $L$, while every mapping class supported on a \emph{proper} subsurface of $U=S-\ell$ acts with uniformly bounded orbits. To this extent, we construct $L$ using the following quasicocycles. Let $g$ be a partial pseudo-Anosov element supported on $U$, which acts loxodromically on $\C U$. By \cite[Lemma 6.16]{DGO}, one can also assume that the maximal virtually cyclic subgroup $E(g)$ containing $g$ surjects onto $\Z$. By a result of Hull-Osin \cite{Hull_Osin}, the acylindricity of the $\MCG(U)$-action on $\C U$ allows one to build a quasicocycle $\psi\colon \Stab{\MCG(S)}{\gamma}\to \R$ for the trivial action (i.e. a \emph{quasimorphism}), such that $\psi$ extends the map $E(g)\to \Z$. The extension should be thought of as a generalisation of \emph{Brooks quasimorphisms} for free groups \cite{Brooks}, as for every $w\in \Stab{\MCG(S)}{\ell}$, $\phi(w)$ roughly counts how many times a geodesic in $\C U$ connecting a basepoint to its $w$-image fellow-travels the translates of the axis of $g$. Then the quasicocycles we use in the proof of \Cref{prop:new_hhg_structures} are linear combinations of $\psi$ and the \emph{Busemann quasicocycle} for the action on the original $\C \ell$, roughly mapping each element to its stable translation length. Different linear combinations then yield different HHG structures, and in turn non-equivalent coarse medians. 

\subsection*{Where do the new coarse medians differ?} 
Let us first stress that the existence of a quasiline in the bottom layer is crucial to our result. For instance, the \emph{pants graph} of a surface, which is quasi-isometric to Teichm\"{u}ller space with the Weil-Petersson metric, admits a HHS structure obtained from that of the associated mapping class group by dropping annular curve graphs. Quite surprisingly, Fioravanti and Sisto recently proved that pants graphs of low-complexity surfaces admit unique coarse median structures \cite{FS_pantsgraph}, and experts believe that this should be true for all surfaces.

The HHG structures built here only differ for the quasilines attached to minimal domains; in particular, the associated coarse medians, which make projections coarse median preserving, all have the same projections to non-minimal domains. Furthermore, in \Cref{cor:new_median_is_the_same_on_DT_RACG} and~\ref{cor:new_median_is_the_same_on_DT_RAAG}, I prove that the new coarse median structures on a RACG (resp. RAAG) restrict to the standard cubical one on products of Dihedral parabolic subgroups (resp. on Abelian parabolic subgroups). For mapping class groups, I suspect that \emph{any} coarse median structure on a mapping class group should restrict to the standard one on Dehn twist flats; I will explore this, and the question on uniqueness of coarse medians on pants graphs, in future work.

As the proof of \Cref{thm:uncountably_many_CM} clarifies, the place where the new coarse medians genuinely differ is the stabiliser of $\ell$. More precisely, in the mapping class group example, consider the subgroup $\langle T,g\rangle\le\Stab{\MCG(S)}{\ell}$, where $g$ is the partial pseudo-Anosov element we used to build the quasimorphisms $\theta_\lambda$. The idea is that, in the structure associated to $\theta_\lambda$, the coarse median of three elements $a,b,c\in \langle T,g\rangle$ is determined by which of $\theta_\lambda(a),\theta_\lambda(b),\theta_\lambda(c)$ projects between the other two; then, for any distinct $\lambda_1\neq\lambda_2$, we can suitably choose $a,b,c$ so that the middle images under $\theta_{\lambda_1}$ and $\theta_{\lambda_2}$ differ. The above is ultimately an instance of the fact, pointed out in e.g. \cite[Remark 4.8]{FS_pantsgraph}, that the product $\C\ell\times\C U\equiv \R\times \C U$ has no canonical coarse median structure, since one can modify the $\R$-projection by adding any unbounded Lipschitz function on $\C U$ (which in our case is roughly $\lambda\psi$).

\subsection*{Comparison to previous work} In \cite{short_HHG:I} I introduced \emph{short HHGs}, a broad class of ``two-dimensional" HHGs including the mapping class group of a five-punctured sphere, RAAGs on triangle- and square-free graphs, many surface-by-surface groups, and more. In the same paper I already explored the idea of replacing the quasilines of the HHG structure to construct non-equivalent coarse medians, and to do so by constructing suitable quasicocycles. However, the tools there crucially use that stabilisers of minimal domains are $\Z$-extensions of hyperbolic groups, both in the construction of the quasicocycles and in that of the HHG structure, so here I expand those techniques to account for the complications arising from ``higher-rank" stabilisers. Furthermore, the advantage of the arguments here is that they are self-contained and require as little background on HHGs as possible; on the other hand, in \cite{short_HHG:I} the construction heavily uses the machinery of \emph{combinatorial HHGs} from \cite{BHMS} and is interlaced with many other ``tweaking" procedures which were needed for applications to Dehn filling quotients in \cite{short_HHG:II}.

\subsection*{Outline of the article} \Cref{sec:background} gathers the background on the interaction between quasilines and quasicocycles (\Cref{sec:qline_from_qmorph}), the construction of quasimorphisms on acylindrically hyperbolic groups (\Cref{sec:qmorph_from_acyl}), hierarchical hyperbolicity (\Cref{sec:HHS}), and coarse medians (\Cref{sec:CM}). \Cref{sec:main_proof} is devoted to the proof of the main result on coarse median structures for HHGs, \Cref{thm:uncountably_many_CM}; the latter follows from the technical \Cref{prop:new_hhg_structures}, which shows how to tweak the HHG structure to insert new quasilines. The requirements of both are collected in \Cref{notation:tweakable_hhg}. In \Cref{sec:applications} we then apply \Cref{thm:uncountably_many_CM} to mapping class groups and RA*Gs, thus proving \Cref{thmintro:mcg},~\ref{thmintro_coxeter}, and~\ref{thmintro_artin}. \Cref{cor:new_median_is_the_same_on_DT_RACG} and~\ref{cor:new_median_is_the_same_on_DT_RAAG} finally show that the new coarse medians on a RA*G coincide with the cubical one when restricted to virtually Abelian parabolic subgroups, and therefore are not induced by any cocompact cubulation.

\subsection*{Acknowledgements}
I thank Elia Fioravanti, Jason Behrstock, Michah Sageev, and Alex Wright for repeatedly encouraging me to expand the tools from \cite{short_HHG:I} to tackle RACGs and all mapping class groups. I am also grateful to Alessandro Sisto for several exchanges on the content of this paper, and for numerous suggestions on how to improve both exposition and results.

\subsection*{AI declaration} No AI was used at any point in the intellectual and writing process, except for literature searches and typo corrections.

\section{Background}\label{sec:background}
\subsection{Group actions on metric spaces}
Throughout the paper, given a group $G$ generated by some subset $X$, we denote the associated Cayley graph by $\Cay{G}{X}$. For a subset $A$ of a metric space $(Z,\dist)$ and a constant $R\ge 0$, we denote the $R$-neighbourhood of $A$ in $Z$ by $N_R(A)$. Recall that a group $G$ acts on a geodesic metric space $(S,\dist)$
\begin{itemize}
    \item coboundedly if there exists $x_0\in S$ and $R\ge 0$ such that $S=\mathcal N_R(G\cdot x_0)$; when we want to emphasise the constant we say the action is $R$-cobounded;
    \item metrically properly if for every $x_0\in S$ and every $K\ge 0$, the set $\{g\in G\mid \dist(x_0,gx_0)\le R\}$ is finite;
    \item geometrically if the action is cobounded and metrically proper.
\end{itemize}
The following is the fundamental insight of geometric group theory: a geometric action on a metric space endows a group with roughly the same geometry.
\begin{lem}[{Milnor-Schwarz, \cite{Milnor,Schwarz}}]\label{lem:MilnorSchwarz}
    Let $G$ be a group acting $R$-coboundedly on a geodesic metric space $(S,\dist)$, and let $x_0\in S$ be a point whose orbit is $R$-dense in $S$. Then the set
    \[X\coloneq \{g\in G\mid \dist(x_0,gx_0)\le 2R+1\}\]
    generates $G$, and the orbit map $g\to gx_0$ is a $G$-equivariant quasi-isometry $\Cay{G}{X}\to S$.
\end{lem}

\subsection{Quasilines from quasicocycles}\label{sec:qline_from_qmorph}
This Subsection recalls the interaction between actions on \emph{quasilines}, i.e. spaces which are quasi-isometric to the real line, and $1$-\emph{quasicocycles}, that is, maps satisfying the quasicocycle relation up to a uniformly bounded error. Most of the material here is extracted from \cite{ABO,NL}.
\begin{defn}\label{defn:1_qc}
    Let $G$ be a group acting on $\R$ by group automorphisms, and let $\sigma\colon G\to \Aut{\R}\cong \Z/2\Z$ be the action. A \emph{$\sigma$-quasicocycle} is a map $f\colon G\to \R$ such that there exists a constant $D$, called the \emph{defect} of $f$, such that, for every $g,h\in G$,
    $$|f(g)+\sigma(g)(f(h))-f(gh)|\le D.$$
    A quasicocycle for the trivial action is called a \emph{quasimorphism}. Notice that, if $f$ is a $\sigma$-quasicocycle and $K=\ker(\sigma)$, which has index at most two in $G$, the restriction of $f$ to $K$ is a quasimorphism. 

    A $\sigma$-quasicocycle $f$ is always bounded on the orbit of any $g\in G-K$, because for every $n\in\mathbb{Z}$ one has
    $$|f(g^{2(2n+1)})|=|f(g^{2(2n+1)})-f(g^{2n+1})+f(g^{2n+1})|\le D.$$
    We say $f$ is \emph{unbounded} if there exists $h\in G$ such that $f|_{\langle h \rangle}$ is unbounded; such a $h$ must necessarily belong to $h$. 
\end{defn}

\begin{rem}[Homogeneisation]\label{defn:homogeneisation}
    A quasimorphism is \emph{homogeneous} if it restricts to a homomorphism on cyclic subgroups. Given any quasimorphism $f\colon K\to \R$ of defect $D$, the limit
    $$\phi(g)=\lim_{n\to+\infty}\frac{f(g^n)}{n}$$
    always exists, and defines a homogeneous quasimorphism of defect at most $2D$ such that $|\phi(g)-f(g)|\le D$ for all $g\in K$ (see e.g. \cite[Lemma 2.21]{calegari}); in particular, a quasimorphism is unbounded if and only if its homogeneisation is non-trivial. Homogeneous quasimorphisms are invariant under conjugation, meaning that $\phi(ghg^{-1})=\phi(h)$ for all $g,h\in K$.
\end{rem}

\begin{ex}[{Busemann quasicocycle, see \cite[Section 4.1]{Manning_actions_on_hyp}}]\label{ex:busemann}
Let $G$ be a group acting on a \emph{quasiline} $L$, that is, a space which is quasi-isometric to $\R$. The action on $\partial L$ gives a homomorphism $\sigma\colon G\to \Z/2\Z$, so let $K$ be its kernel. Given a sequence $\{x_n\}\in L$ converging to some $\xi\in \partial L$, the map $\beta_{\{x_n\}}\colon K\to \R$ defined by
  $$\beta_{\{x_n\}}(g)=\limsup_{n\to \infty} \dist(gx_0,x_n)-\dist(x_0,x_n),$$
is a quasimorphism. Its homogeneisation $\beta$ does not depend on the choice of the sequence, and has the property that an element $g\in K$ acts loxodromically on $L$ if and only if $\beta(g)\neq 0$. Furthermore, \cite[Proposition 2.10]{NL} shows that $\beta$ extends to a $\sigma$-quasicocycle $\phi\colon G\to \R$ by fixing $i\in G-K$ and setting $\phi(ki^n)=\beta(k)$ for every $k\in K$ and $n=0,1$. We stress that the extension is not canonical, though this will not matter for our purposes. 
\end{ex}

\noindent The following, which is essentially \cite[Lemma 2.14]{NL}, should be thought of as a converse of \Cref{ex:busemann}: up to bounded distance, every quasicocycle is the Busemann quasicocycle for the action on some quasiline. 
\begin{lem}\label{lem:quasiline_from_qc}
Let $\sigma\colon G\to \Z/2\Z$ be a homomorphism with kernel $K$, and let $\phi\colon G\to \R$ be an
unbounded $\sigma$-quasicocycle, which restricts to a homogeneous quasimorphism on $K$. Let $C\ge 0$ be
such that $D(\phi)\le C/2$ and $\phi(K)\cap (0,C/2)\neq \emptyset$; if $K\lneq G$ suppose also that $\phi(G-K)\cap [0,C/2)\neq \emptyset$. Let $X \coloneq \{g \in G \mid|\phi(g)| < C\}$. Then $X$ generates $G$, and the map $\phi\colon \Cay{G}{X}\to \R $ is a quasi-isometry. Furthermore:
\begin{enumerate}
            \item \label{item:qc1} $g\in G$ acts loxodromically on $\Cay{G}{X}$ if and only if $\phi$ is unbounded on $\langle g\rangle$.
            \item \label{item:qc2} $g\in G$ fixes the ideal endpoints of $\Cay{G}{X}$ if and only if $g\in K=\ker(\sigma)$.
            \item \label{item:qc3} If $H\le K$ is such that $|\phi|_{H}|$ is bounded, then $H$ acts on $\Cay{G}{X}$ with uniformly bounded orbits.
        \end{enumerate}
\end{lem}

\noindent The assumption that $\phi$ is homogeneous on $K$ is not restrictive, as up to bounded distance one can always replace $\phi|_K$ by its homogeneisation while keeping $\phi|_{G-K}$ the same.

        

\begin{proof}
    The main statement is \cite[Lemma 2.14]{NL}. For the ``furthermore" part, \Cref{item:qc1} follows immediately from the fact that $\phi$ is a quasi-isometry. Moving to \Cref{item:qc2}, let $g\in G$, let $k\in K$ such that $\phi|_{\langle k\rangle}$ is unbounded, and suppose without loss of generality that $\lim_{n\to +\infty} \phi(k^n)=+\infty$. Then
    \[\lim_{n\to+\infty} \phi(gk^n)=\lim_{n\to+\infty} \phi(gk^ng^{-1})=\lim_{n\to+\infty} \sigma(g)\psi(k^n)=\sigma(g)\infty,\]
    where for both equalities we used that $f$ is a $\sigma$-quasicocycle. Finally, if $H\le K$ and $\phi|_H$ is bounded, then for every $g\in G$ we also have that $\phi|_{gHg^{-1}}$ is bounded, as $\phi$ is a quasicocycle, and therefore  $\phi|_{gHg^{-1}}\equiv 0$ as $\phi$ restricts to a homogeneous quasimorphism on $K$. Therefore every conjugate of $H$ lies in $X$ by definition of the latter; in turn this proves \Cref{item:qc3}, since \[\diam_{\Cay{G}{X}}(H\cdot g)= \diam_{\Cay{G}{X}}(g^{-1}H\cdot g)\le 1.\qedhere\]
\end{proof}

\noindent If one starts with an action on a quasiline, the new quasiline associated to the quasimorphism is essentially the same as before:
\begin{cor}\label{cor:quasiline_from_Busemann_qc}
    Let $G$ be a group acting on a quasiline $(S,\dist)$ with a loxodromic element, and let $K$ be the kernel of the action on $\partial S$. Let $\phi\colon G\to \R$ be the Busemann quasicocycle for the action (associated to a choice of element $i\in G-K$), and let $L=\Cay{G}{X}$ be the quasiline from \Cref{lem:quasiline_from_qc} induced by $\phi$. Then $L$ and $S$ are $G$-equivariantly quasi-isometric.
\end{cor}

\begin{proof}
    Let $L\to S$ be the orbit map $g\to g\cdot x_0$, which is $G$-equivariant by construction and coarsely surjective because $S$ is a quasiline and $G$ contains a loxodromic element. It is therefore enough to show that the orbit map is a quasi-isometry; for the ease of notation, we shall write $A\asymp B$ to mean that there exist constants $K\ge 1,C\ge 0$ such that $B/K-C\le A\le BK+C$.

    We firstly prove that the orbit map restricts to a quasi-isometry on $K$. Recall that the Busemann quasimorphism on $k\in K$ is the homogeneisation of 
    \[\beta_{\{x\}}(k)=\limsup_{n\to+\infty}\dist(k\cdot x_0, x_n)-\dist(x_0,x_n),\]
    where $\{x_n\}_{n\in \N}$ is any sequence converging to a boundary point $\xi$. Notice that for every $n$ one has
    \[\dist(k\cdot x_0, x_0)\asymp |\dist(k\cdot x_0, x_n)-\dist(x_0, x_n)|\]
    since $S$ is a quasiline. Thus, passing to the limsup or liminf (depending on whether $\beta_{\{x_n\}}(k)$ is positive or negative), we get that
    \[\dist(k\cdot x_0, x_0)\asymp |\beta_{\{x_n\}}(k)|\asymp |\phi(k)|\asymp |k|_X,\]
    where we used the the definition of $\phi|_K$ as the homogeneisation of $\beta_{\{x_n\}}$, and the fact that $\phi\coloneq \Cay{G}{X}\to \R$ is a quasi-isometry. Now, if $g=ki^n$ for $n\in \{0,1\}$, then the triangle inequality plus the fact that $G$ acts isometrically on $S$ yield
    \[\dist(k\cdot x_0, x_0) -\dist(s\cdot x_0, x_0)\le \dist(g\cdot x_0, x_0)\le \dist(k\cdot x_0, x_0) +\dist(i\cdot x_0, x_0),\]
    so
    \[\dist(g\cdot x_0, x_0)\asymp \dist(k\cdot x_0, x_0)\asymp |k|_X.\]
    In turn, 
    \[|k|_X-|i|_X\le |g|_X\le |k|_X+|i|_X\]
    hence $|k|_X\asymp |g|_X$. This shows that $\dist(g\cdot x_0, x_0)\asymp |g|_X$, i.e. the orbit map is a quasi-isometric embedding. 
\end{proof}
\subsection{Quasimorphisms from acylindrical actions}\label{sec:qmorph_from_acyl} Here we recall how to extend quasicocycles from virtually cyclic subgroups of acylindrically hyperbolic groups.

\begin{defn}[\cite{Osin_AH}]
A group $G$ is \emph{acylindrically hyperbolic} if it is not virtually cyclic, it acts with unbounded orbits on a Gromov-hyperbolic metric space $(S,\dist)$, and the action is \emph{acylindrical}, i.e. for every $R\ge 0$ there exist $N>0,L>0$ such that for every $x,y\in S$ with $\dist(x,y)\geq L$ one has
    \[\#\{g\in G\mid \dist(x,gx)\leq R,\,\dist (y,gy)\leq R\}\leq N.\]
\end{defn}

\noindent Osin also proved that acylindrical hyperbolicity is equivalent to the existence of a non-trivial \emph{hyperbolically embedded subgroup}. Rather than giving the general definition of the latter, which can be found in \cite[Definition 4.25]{DGO}, we recall an equivalent characterisation in the case when the subgroup is hyperbolically embedded in a hyperbolic Cayley graph, which was independently proven in \cite[Theorem 3.9]{AMS} and \cite[Theorem 4.13]{Hull_small_cancellation}.
\begin{defn}
    Let $G$ be a group, and let $X$ be a generating set such that $\Gamma\coloneq \Cay{G}{X}$ is hyperbolic. A subgroup $H\le G$ is \emph{hyperbolically embedded} in $(G,X)$, denoted $H\hookrightarrow_h (G,X)$, if
    \begin{itemize}
        \item $H$ is finitely generated by a subset $Y$;
        \item $H$ is quasi-isometrically embedded: there exist constants $L,c$ such that, for all $h\in H$, $|h|_Y\le L|h|_{X}+c$;
        \item $H$ is \emph{geometrically separated}: for every $\varepsilon> 0$ there exists $R\ge 0 $ such that, for every $g\in G$, if $\diam_\Gamma(H\cap N_\varepsilon(gH))\ge R$, then $i=j$ and $g\in H$.
    \end{itemize}
\end{defn}

\noindent We next recall a natural class of hyperbolically embedded subgroups, induced by loxodromic elements in cobounded, acylindrical actions.
\begin{ex}\label{ex:loxoWPD_is_HH} Let $G$ be a non-virtually cyclic group, let $\Cay{G}{X}$ be hyperbolic and unbounded, and suppose that the action is acylindrical. Let $g\in G$ act loxodromically on $\Cay{G}{X}$, which exists by \cite[Theorem 1.1]{Osin_AH}, and let $E(g)$ be the maximal virtually cyclic subgroup of $G$ containing $g$, which exists by \cite[Lemma 6.5]{DGO}. Then $E(g)\hookrightarrow_h(G,X)$ by \cite[Theorem 6.8]{DGO}, whose proof essentially checks geometric separation. 
    
    For completeness, we point out that the above example can be widely generalised to non-cobounded, non-acylindrical actions as well, provided $g$ satisfies the \emph{WPD property} which, roughly, states that the action is acylindrical ``along the axis of $g$" \cite[Definition 6.1]{DGO}. 
\end{ex}

\begin{lem}\label{lem:qmorph_from_acylindrical}
Let $G$ be a non-virtually cyclic group, let $X$ be such that $\Cay{G}{X}$ is unbounded and hyperbolic, and suppose the $G$-action on $\Cay{G}{X}$ is acylindrical. There exists a non-trivial homogeneous quasimorphism $Q\colon G\to \R$ which is coarsely Lipschitz with respect to $X$, meaning that there exists $L\ge 0$ such that, for all $g,h\in G$, $|Q(g)-Q(h)|\le L|gh^{-1}|_X+L$.
\end{lem}

\begin{proof}
    Let $t\in G$ be a loxodromic element for the action on $\Cay{G}{X}$ such that $E(t)$ surjects onto $\Z$, which exists by e.g. \cite[Lemma 6.16]{DGO}. The surjection $q\colon E(t)\to \Z$, seen as a homogeneous quasimorphism $E(t)\to \R$, can be extended to a homogeneous quasimorphism $Q\colon G\to \R$ by a result of Hull and Osin \cite[Theorem 1.4]{Hull_Osin}, later generalised by Frigerio, Pozzetti, and Sisto \cite[Theorem 1]{FPS} to higher dimensional quasicocycles. Our goal is to go through the latter construction to show that $Q$ is actually coarsely Lipschitz with respect to $X$. 

    We recall some notation from \cite{FPS}, see also \cite[Chapter 9]{BCSV} for a gentle walk-through.
    Endow $\Cay{G}{X}$ with the usual path metric where edges have length one, and let $(\hat G,\hat \dist)$ be the metric graph obtained from $\Cay{G}{X}$ by adding a vertex $c(B)$ for every coset $B$ and connecting every $x\in B$ to $c(B)$ via an edge of length $1/4$. Let $\dist_{E({t})}\colon E(t)\times E(t)\to [0,+\infty]$ be the \emph{relative metric} from \cite[Definition 4.25]{DGO}; for our purposes, it will suffice to know that $\dist_{E(t)}$ is $E(t)$-left-invariant and bounded from above by $\dist_{\Cay{G}{X}}$. For every $g\in G$ one then sets $\dist_{gE(t)}$ as the $g$-translate of $\dist_{E(t)}$; this is well-defined by left-invariance of $\dist_{E(t)}$. Define $\pi_B\colon G\to B$ by sending $x\in G$ to the set of all $y\in B$ realising the distance $\hat\dist(x,B)$. 
    \begin{claim}
        There exists $D\ge 1$ such that, if $g,h\in G$ satisfy $\dist_{\Cay{G}{X}}(g,h)\le 1$, then the diameter of $\pi_B(g)\cup \pi_B(h)$, with respect to $\dist_B$, is at most $D$. 
    \end{claim}
    \begin{claimproof} \cite[Lemma 2.8]{FPS} grants the existence of a constant $D\ge 1$ such that, if the diameter of $\pi_B(g)\cup \pi_B(h)$ is at least $D$, then every $\hat G$-geodesic between $g$ and $h$ passes through $c(B)$. If $g=h$ we are done by the Lemma; hence suppose $g$ and $h$ are $X$-adjacent, and therefore $\hat\dist(g,h)\le 1$. By the Lemma, if the diameter of their projections was greater than $D$, then the $X$-edge between them would not be a geodesic, so $\hat\dist(g,h)$ would have to be $1/2$; as a consequence, the unique geodesic between $g$ and $h$ would have to be the path $\{g,c(B),h\}$, meaning that $g,h$ would belong to $B$. However this is impossible, as then $g=\pi_B(g)$, $h=\pi_B(h)$, and $\dist_B(g,h)\le \dist_{\Cay{G}{X}}(g,h)=1$ as they are $X$-adjacent. 
    \end{claimproof}
    \noindent Now, the homogeneous quasimorphism $Q\colon G\to \R$ extending $q$ is the homogeneisation of the quasimorphism $Q_0$ sending each $g\in G$ to
    \[Q_0(g)\coloneq \sum\frac{1}{|\pi_B(1)\cup \pi_B(g)|}\sum_{x\in \pi_B(1),\,y\in\pi_B(g)}q(x^{-1}y), \]
    where the sum is over all $B\in G/E(t)$ such that $\diam_B(\pi_B(1)\cup \pi_B(g))>2D$. Notice that $Q_0$ is trivial on $X$ by the Claim; in turn, by \Cref{defn:homogeneisation}, we get $\sup_{x\in X}|Q(x)|\le \delta$, where $\delta$ is the defect of $Q_0$. Hence, for every $g,h\in G$, we have
    \begin{align*}
        |Q(g)-Q(h)|=|Q(g)+Q(h^{-1})|&\le 2\delta+ |Q(gh^{-1})|\\
        &\le 2\delta+|gh^{-1}|_X\left(2\delta+\sup_{x\in X}|Q(x)|\right)\\&\le 2\delta+3\delta|gh^{-1}|_X,
    \end{align*}
    where we repeatedly used that the the defect of $Q$ is $2\delta$ by \Cref{defn:homogeneisation}. If we set $L=3\delta$, the above inequality concludes the proof.
\end{proof}

\noindent We point out an immediate consequence of \Cref{lem:qmorph_from_acylindrical} that will be useful later:
\begin{cor}\label{cor:bounded_orbit_implies_trivial}
    In the setting of \Cref{lem:qmorph_from_acylindrical}, if $H\le G$ has a bounded orbit in $\Cay{G}{X}$, then $Q|_H\equiv 0$. 
\end{cor}

\begin{proof}
    Suppose that $\diam_{\Cay{G}{X}}(Qg)<T$ for some $g\in G$. By left-invariance of the Cayley metric, $\diam_{\Cay{G}{X}}(g^{-1}Qg)<T$, so for every $h\in Q$ we have that 
    \[Q(h)=Q(g^{-1}hg)\le LT+L.\]
    where we used that $Q$ is homogeneous and therefore invariant under conjugation. But then, again by homogeneity, $Q(h)=\lim_{n\to+\infty}\frac{Q(h^n)}{n}=0$, as required.
\end{proof}

\subsection{Hierarchical hyperbolicity}\label{sec:HHS}
This subsection gathers definitions and properties of hierarchically hyperbolic spaces and groups, first introduced in \cite{HHS_II}; see \cite[Definition~2.8]{russell} for this formulation of some of the axioms, which has by now become customary.

\begin{defn}[Hierarchically hyperbolic space]\label{defn:HHS}
    Let $E>0$, and let $Z $ be an $(E,E)$--quasigeodesic space.  A \emph{hierarchically hyperbolic space (HHS) structure with constant $E$} for $Z $ is an index set $\frakS$ and a set $\{ \C W \mid W\in\frakS\}$ of $E$-hyperbolic metric spaces $(\C W,\dist_W)$ such that the following axioms are satisfied. 
    \begin{enumerate}[label=(\arabic*{})]
        \item\label{axiom:projections} \textbf{(Projections.)} For each $W \in \frakS$, there exists a \emph{projection} $\pi_W \colon Z  \rightarrow 2^{\C W}$  that is a $(E,E)$--coarsely Lipschitz, $E$--coarsely onto, $E$--coarse map.
	
        \item\label{axiom:nesting} \textbf{(Nesting.)} If $\frakS \neq \emptyset$, then $\frakS$ is equipped with a  partial order $\nest$ and contains a unique $\nest$--maximal element. When $V\nest W$, we say $V$ is \emph{nested} in $W$. For each $W\in\frakS$, we denote by $\frakS_W$ the set of all $V\in\frakS$ with $V\nest W$.  Moreover, for all $V,W\in\frakS$ with $V\propnest W$ there is a specified non-empty subset $\rho^V_W\subseteq \C W$ with $\diam(\rho^V_W)\le E$.
            
        \item\label{axiom:finite_complexity}  \textbf{(Finite complexity.)} Any set of pairwise $\nest$--comparable elements has cardinality at most $E$.
            
        \item\label{axiom:orthogonal} \textbf{(Orthogonality.)} The set $\frakS$ has a symmetric relation called \emph{orthogonality}. If $V$ and $W$ are orthogonal, we write $V\perp W$ and require that $V$ and $W$ are not $\nest$--comparable. Further, whenever $V\nest W$ and $W\perp U$, we require that $V\perp U$. We denote by $\frakS_W^\perp$ the set of all $V\in \frakS$ with $V\perp W$.
            
        \item\label{axiom:containers}\textbf{(Containers.)} For each $W \in \frakS$ and $U \in \frakS_W$ with $ \frakS_W\cap \frakS_U^\perp \neq \emptyset$, there exists $Q \in\frakS_W-\{W\}$ such that $V \nest Q$ whenever $V \in\frakS_W \cap \frakS_U^\perp$.
		
        \item\label{axiom:transversality} \textbf{(Transversality.)} If $V,W\in\frakS$ are neither orthogonal nor $\nest$-comparable, we say $V$ and $W$ are \emph{transverse}, denoted $V\trans W$.  Moreover, for all $V,W \in \frakS$ with $V\trans W$, there are non-empty sets $\rho^V_W\subseteq \C W$ and $\rho^W_V\subseteq \C  V$, each of diameter at most $E$.
		
        \item\label{axiom:consistency} \textbf{(Consistency.)} For all $x \in Z $ and $V,W,U \in\frakS$:
	\begin{itemize}
            \item if $V\trans W$, then $\min\left\{\dist_{W}(\pi_W(x),\rho^V_W),\dist_{V}(\pi_V(x),\rho^W_V)\right\}\le E$, and
            \item if $U\nest V$ and either $V\propnest W$ or $V\trans W$ and $W\not\perp U$, then $\dist_W(\rho^U_W,\rho^V_W)\le E$.
	\end{itemize}
		
        \item\label{axiom:bounded_geodesic_image} \textbf{(Bounded geodesic image.)} For all $V,W\in\frakS$ and for all $x,y \in Z $, if  $V\propnest W$ and $\dist_V(\pi_V(x),\pi_V(y)) \geq E$, then every $\C W$--geodesic from $\pi_W(x)$ to $\pi_W(y)$ must intersect $\mathcal{N}_E(\rho_W^V)$.
		
        \item\label{axiom:partial_realisation} \textbf{(Partial realization.)} If $\{V_i\}$ is a finite collection of pairwise orthogonal elements of $\frakS$ and $p_i\in \C V_i$ for each $i$, then there exists $x\in Z $ so that:
	\begin{itemize}
		\item $\dist_{V_i}(\pi_{V_{i}}(x),p_i)\le E$ for all $i$;
            \item for each $i$ and each $W\in\frakS$, if $V_i\propnest W$ or $W\trans V_i$, we have $\dist_{W}(\pi_W(x),\rho^{V_i}_W)\le E$.
	\end{itemize}

        \item\label{axiom:uniqueness} \textbf{(Uniqueness.)} There is a function $\theta \colon [0,\infty) \to [0,\infty)$ so that for all $r \geq 0$, if $x,y\in Z $ and $\dist_Z (x,y)\geq\theta(r)$, then there exists $W\in\frakS$ such that $\dist_W(\pi_W(x),\pi_W(y))\geq r$. 

        \item\label{axiom:large_link_lemma} \textbf{(Large links.)} For all $W\in\frakS$ and  $x,y\in Z $, there exists $\{V_1,\dots,V_m\}\subseteq\frakS_W -\{W\}$ such that $m$ is at most $E \dist_{W}(\pi_W(x),\pi_W(y))+E$, and for all $U\in\frakS_W - \{W\}$, either $U\in\frakS_{V_i}$ for some $i$, or $\dist_{U}(\pi_U(x),\pi_U(y)) \le E$.  
    \end{enumerate}
    We use $\frakS$ to denote the relative HHS structure, including the index set $\frakS$, spaces $\{\C W : W \in \frakS\}$, projections $\{\pi_W : W \in \frakS\}$, and relations $\nest$, $\perp$, $\trans$.  We call an element $U\in \frakS$ a \emph{domain}, the associated space $\mathcal CU$ its \emph{coordinate space}, and call the maps $\rho_W^V$ the \emph{relative projections} from $V$ to $W$. We say a domain $U$ is \emph{unbounded} if $\mathcal C U$ is unbounded. The number $E$ is called the \emph{hierarchy constant} for $\frakS$; notice that every $E'\ge E$ is again a hierarchy constant for $\frakS$, so we are often free to enlarge $E$ by a bounded amount.  We often suppress reference to the projection maps, so for every $x,y\in Z$ and $U\in \frakS$ we write $\dist_U(x,y)$ to mean $\dist_U(\pi_U(x), \pi_U(y))$.
\end{defn}	

\noindent The large links axiom (\Cref{defn:HHS}.\ref{axiom:large_link_lemma}) can be replaced with the following, which is traditionally called \emph{passing up}. The equivalence of the two statements is \cite[Lemma~2.5]{HHS_II} for one direction, and \cite[Section~4.8]{Durham_Cubulating_Infinity} or \cite[Lemma~5.3]{PetytSpriano_Eyries} for the other.
\begin{enumerate}[label=(\arabic*{}'), start=11]
    \item\label{axiom:passing_up} \textbf{(Passing up).} For every $t>0$, there exists an integer $P=P(t)>0$ such that if $V\in \frakS$ and $x,y\in Z$ satisfy $\dist_{U_i}(x,y)>E$ for a collection of domains $\{U_i\}_{i=1}^P$ with $U_i\in \frakS_V$, then there exists $W\in \frakS_V$ with $U_i\propnest W$ for some $i$ such that $\dist_W(x,y)>t$. 
\end{enumerate}

\begin{defn}[Hierarchically hyperbolic group]\label{defn:HHG}
Let $G$ be a finitely generated group. We say $G$ is a \textit{hierarchically hyperbolic group} (HHG) if the following hold.
    \begin{enumerate}[label=(\roman*{})]
        \item\label{HHG_structure} There exists a HHS $(Z,\frakS)$ such that $G$ acts metrically properly and coboundedly on $Z$. 
        \item\label{HHG_action} There is a $\nest$--, $\perp$--, and $\trans$--preserving action of $G$ on $\frakS$ by bijections such that $\frakS$ contains finitely many $G$--orbits.
        \item\label{HHG_isometries} For each $W\in \frakS$ and $g\in G$, there exists an isometry $g_W\colon \C W \to \C (gW)$ satisfying the following for all $V,W\in\frakS$, $h\in G$, and $x\in Z$.
            \begin{itemize}
                \item The maps $(gh)_W\colon \C W \to \C (ghW)$ and $g_{hW}\circ h_W\colon \C W\to \C (ghW)$ coincide.
                \item $g_W(\pi_W(x))=\pi_{gW}(g\cdot x)$.
                \item If $V\trans W$ or $V\propnest W$, then $g_W(\rho^V_W)=\rho^{gV}_{gW}$.
            \end{itemize}
    \end{enumerate}
In \Cref{HHG_isometries} we often drop the subscript on the isometry $g_W$ and simply write $g$ when the domain $W$ is clear from context. 
\end{defn}
\begin{rem}\label{rem: fi_sgr_is_HHG}
Notice that, if $G$ is a HHG with underlying HHS $(Z,\frakS)$ and $H$ is a finite-index subgroup of $G$, then $H$ is itself a HHG, with the same underlying HHS.
\end{rem}

\begin{rem}[Bounded domain dichotomy]\label{rem:BDD}
    Since there are finitely many orbits of domains in a HHG, the latter satisfies the \emph{bounded domain dichotomy}: there exists a constant $C\ge 0$ such that, for every $U\in \frakS$, either $\diam\C U\le C$ or $U$ is unbounded. When working with a HHG, we therefore can and will assume that the hierarchy constant $E$ is greater than the diameter of any bounded domain.
\end{rem}

\subsubsection{Useful tools}
We gather here some spurious facts about HHSs and HHGs that we shall need later. 

\begin{lem}[{\cite[Lemma 2.1]{HHS_II}}]
    \label{lem:pairwise_orth_are_at_most_E} If $V_1,\ldots,V_k\in \frakS$ are pairwise orthogonal, then $k\le E$.
\end{lem}

\begin{lem}[{\cite[Lemma 1.5]{DHS}}]\label{lem:DHS}
    Let $U,V,W\in\frakS$. If $U\orth V$, and both $\rho^U_W$ and $\rho^V_W$ are defined, then $\dist_{W}(\rho^U_W,\rho^V_W)\le 2E$.
\end{lem}

\begin{lem} \label{lem:acyl_action_on_CS} A HHG $(G,\frakS)$ acts acylindrically on the top-level coordinate space $\C S$. As a consequence, if $G$ is not virtually cyclic and $\C S$ is unbounded, then $G$ is acylindrically hyperbolic.
\end{lem}



\noindent We also list some mild properties that most natural HHSs enjoy. The first is a strengthening of Axiom \ref{axiom:containers}:
\begin{defn}[Clean containers]\label{defn:CC}
    A HHS $(Z,\frakS)$ has \emph{clean containers} if, for every $U\propnest V$ $\frakS_V\cap\frakS_U^\orth$ is either empty or it admits a unique maximal element $Q$, called the \emph{orthogonal complement} of $U$ in $V$. If $V=S$, we denote the orthogonal complement of $U$ in $S$ by $U^\orth$.
\end{defn}


\noindent For the other additional property, we need to recall the definition of product regions.

\begin{defn}[Product regions and gates]\label{defn:prodreg}
Let $\mathcal U=\{U_1,\ldots, U_k\}\subseteq \frakS$ be a collection of pairwise orthogonal domains. For a domain $W\in\frakS$, set $\mathcal U\propnest W$ (resp. $\mathcal U\trans W$) if there exists $U_i\in \mathcal U$ such that $U_i\propnest W$ (resp. $U_i\trans W$). Then define the \emph{product region} associated to $\mathcal U$ as the subspace
\[P_{\mathcal U}=\left\{x\in Z \mid \dist_W\left(x,\bigcup\rho^{U_i}_W\right)\le E\mbox{ whenever } \mathcal U\trans W \mbox{ or }\mathcal U\propnest W\right\},\]
where the union is over all $U_i$ for which $\rho^{U_i}_W$ is well-defined. When $\mathcal U$ consists of a single domain $U$, we denote the associated product region simply by $P_U$. 
\end{defn}

\begin{rem}[Gate map]
By combining \cite[Lemma 5.5 and Construction 5.10]{HHS_II}\footnote{A note for experts: for the simplified definitions of $P_U$ and $\gate_U$ we give here to be equivalent to the ones from \cite{HHS_I}, one should first assume $E$ is sufficiently large that the tuple $\{\pi_W(x)\}_{W\nest U\mbox{ or }W\orth U}\cup \{\rho^U_W\}_{W\trans U\mbox{ or } U\propnest W}$ is $E$-consistent for all $x\in Z$, and can be $E$-coarsely realised. Since we will always deal with HHGs, which have finitely many $G$-orbits of domains, we can once and for all replace $E$ by a large enough constant, and forget about these technicalities.} there exists a coarsely Lipschitz map $\gate_U\colon Z\to P_U$, which maps a point $x\in Z$ to the coarsely unique point $\gate_U(x)\in P_U$ such that:
\begin{itemize}
    \item For all $W\nest U$ or $W\orth U$, $\dist_{W}(x,\gate_U(x))\le E$;
    \item For all $U\propnest W$ or $W\trans U$, $\dist_{W}(\rho^U_W,\gate_U(x))\le E$.
\end{itemize}
\end{rem}


\noindent The following very mild requirement identifies a product region in a HHG and the associated domain stabiliser:
\begin{defn}[Cobounded product regions]\label{defn:cobounded_pr}
    A HHG has \emph{cobounded product regions} if, for every $U\in\frakS$, $\Stab{G}{U}$ acts coboundedly on $P_U$. In particular, since the whole $G$ acts properly on $Z$, $\Stab{G}{U}$ is quasi-isometric to $P_U$; the gate onto $P_U$ then yields a $\Stab{G}{U}$-equivariant, coarsely Lipschitz coarse retraction $Z\to \Stab{G}{U}$, which we still call $\gate_U$ with a little abuse of notation.
\end{defn}

\subsection{Coarse median structures}\label{sec:CM}
We conclude the background section by recalling the notion of a coarse median structure, and how it arises from a hierarchically hyperbolic structure.

\begin{defn}[Coarse median structure]\label{defn:coarse_median}
A \emph{coarse median space} is a metric space $(Z,\dist)$ equipped with a symmetric ternary operation $\mu\colon Z^3\to Z$, called a \emph{coarse median}, satisfying the following, for some constant $R\ge 0$:
\begin{itemize}
  \item \emph{Localisation}: $\mu(x,x,y)=a$ for all $x,y\in  Z $;
 \item \emph{Coarsely Lipschitz}: for all $x,x',y,z\in  Z $, we have
      $$d(\mu(x,y,z), \mu(x',y,z)) \leqslant Rd(x,x')+R;$$
  \item \emph{Coarse associativity}: for all $x,y,z,p\in  Z $, we have
\[
\dist\left(\mu\left(\mu(x,p,y),p,z\right), \mu\left(x,p,\mu(y,p,z)\right)\right)\le R.
\]
 \end{itemize}
 Two coarse median $\mu_1,\mu_2\colon  Z ^3\to  Z $ are \emph{equivalent}, and we write $\mu_1\sim \mu_2$, if \[\sup_{x,y,z\in  Z } \dist_{ Z }(\mu_1(x,y,z),\mu_2(x,y,z))<+\infty;\] an equivalence class of coarse medians is called a \emph{coarse median structure}. 
 
 Let $G$ be a finitely generated group, and let $Z$ be a Cayley graph for $G$ with respect to a finite generating set. A coarse median structure on $Z$ is \emph{$G$-equivariant} if, given any representative $\mu$ in the equivalence class, $\mu \sim g\mu(g^{-1}(\cdot),g^{-1}(\cdot),g^{-1}(\cdot)) $; a group equipped with an equivariant coarse median structure is called a \emph{coarse median group}.
\end{defn}

\begin{ex}[Hyperbolic spaces admit coarse median structures]
    A hyperbolic space $Y$ admits a coarse median structure: for every $a,b,c\in Y$, their coarse median is the coarse centre of any geodesic triangle with these vertices (see e.g. \cite[Proposition III.$\Gamma$.1.17]{bridsonhaefliger}). This coarse median structure is also unique by \cite[Theorem 4.2]{NWZ}.
\end{ex}

\begin{ex}[HHSs admit coarse median structures]\label{ex:coarse_median_from_HHS}
By e.g. \cite[Theorem 7.3]{HHS_II}, which in turn builds on an observation of Bowditch \cite{bowditch_mcg}, a HHS $(Z,\frakS)$ admits a coarse median structure $\mu$, obtained as follows: for every $x,y,z\in Z$ and every $U\in\frakS$, the $U$-coordinate of $\mu(x,y,z)$ coarsely agrees with the coarse median of $\pi_U(x),\pi_U(y),\pi_U(z)$ in the hyperbolic space $\C U$. In other words, the coarse median associated to $(Z,\frakS)$ makes all coordinate projections $\pi_U\colon Z\to \C U$ coarse median preserving, and is characterised by this property by the uniqueness axiom \Cref{defn:HHS}.\eqref{axiom:uniqueness}; as we will prove in \Cref{thm:uncountably_many_CM}, different HHS structures can produce different coarse median structures. Notice also that, if $(Z,\frakS)$ is a HHG structure for a group $G$, then $\mu$ is $G$-equivariant, by how $G$ acts on coordinate spaces.
\end{ex}

\noindent A subspace which is "coarsely" closed under the coarse median inherits a coarse median from the ambient space. 

\begin{defn}
    Let $(Z,\dist, \mu)$ be a coarse median space. A subset $Y\subseteq Z$ is \emph{median-quasiconvex}, or quasiconvex with respect to $\mu$, if there exists a constant $K\ge 0$ such that, for every $y,y'\in Y$ and $z\in Z$, $\mu(y,y',z)\in N_R(Y)$. If $(G,\mu)$ is a coarse median group and $Z$ is a Cayley graph for $G$ with respect to a finite generating set, we say that a finitely generated subgroup $H$ is median-quasiconvex if the orbit $H\cdot g\subseteq Z$ is median-quasiconvex, for some choice of $g\in G$ (hence every choice, by e.g. \cite[Lemma 2.3(6)]{FS_pantsgraph}).
\end{defn}

\begin{lem}
    A median-quasiconvex subset inherits a well-defined coarse median structure. In the case of a median-quasiconvex subgroup, the inherited coarse median structure is equivariant.
\end{lem}

\begin{proof}
    Let $Y\subseteq Z$ be $K$-quasiconvex with respect to $\mu$. For all $y_1,y_2,y_3\in Y$, define $\mu_Y(y_1,y_2,y_3)$ as any point in $Y$ within distance $K$ from $\mu(y_1,y_2,y_3)$. We can assume that $\mu_Y\coloneq Y^3\to Y$ is symmetric, and that if $y_1=y_2$ then $\mu_Y(y_1,y_1,y_3)=y_1=\mu(y_1,y_1,y_3)$; in the case where $Y$ is a subgroup of a coarse median group, one can also choose $\mu_Y$ to be coarsely equivariant with respect to the $Y$-action. Any two possible choices for $\mu_Y(y_1,y_2,y_3)$ are at distance at most $2K$, so $\mu_Y$ is coarsely well-defined. Moreover, $\mu_Y$ is a coarse median on $(Y,\dist)$. Indeed, we chose $\mu_Y$ to satisfy symmetry and localisation; the coarsely Lipschitz property was already satisfied by $\mu$; and coarse associativity is a combination of the coarsely Lipschitz property and coarse associativity for $\mu$. 
\end{proof}

\section{Uncountably many coarse median structures}\label{sec:main_proof}
In this Section we prove that numerous HHGs admit a continuum of coarse median structures, \Cref{thm:uncountably_many_CM}. We shall work under the following assumptions.
\begin{assumption}\label{notation:tweakable_hhg}
    Let $G$ be a HHG, with underlying HHS $(Z,\frakS)$. We assume the following facts hold:
    \begin{enumerate}
        \item \label{need:prodregions} $(Z,\frakS)$ has clean containers (\Cref{defn:CC}) and cobounded product regions (\Cref{defn:cobounded_pr}).
        \item \label{need:annular_curve_graph} There exists a minimal domain $\ell\in \frakS$ whose coordinate space $\C \ell$ is a quasiline.
        \item \label{need:orth_compl} The orthogonal complement $\ell^\orth$ is unbounded, it is nested in no unbounded domain possibly except $S$, and $(\ell^\orth)^\orth=\ell$. Furthermore, for any $W\in \frakS$, if $\ell\trans W$ then $\ell^\orth\trans W$. 
        \item \label{need:dehntwist} There exists $T\in \Stab{G}{\ell}$ such that $\langle T\rangle$ is normal in $\Stab{G}{\ell}$, it acts loxodromically on $\C\ell$ and trivially on $\frakS_\ell^\orth$, and for every $W\in \frakS_\ell^\orth$ the orbits of $\langle T\rangle$ in $\C W$ are uniformly bounded. 
        \item \label{need:acylindricity} $H\coloneq \Stab{G}{\ell}/\langle T\rangle $ is not virtually cyclic, and the induced $H$-action on $\C \ell^\orth/\langle T\rangle$ is acylindrical.
    \end{enumerate}
\end{assumption}



\noindent The technical heart of the proof of  \Cref{thm:uncountably_many_CM} will be to show that the coordinate space for $\ell$ can be freely replaced by another quasiline, chosen from the following family.

\begin{construction}\label{construction:newline} Firstly, by \Cref{lem:MilnorSchwarz}, we once and for all replace $Z$ by a Cayley graph for $G$, with respect to some finite generating set. Next, since $(G,\frakS)$ has cobounded product regions, $\Stab{G}{\ell}$ acts coboundedly on $\C \ell^\orth$, so $H$ acts coboundedly on $\C \ell^\orth/\langle T\rangle$. By the Milnor-Schwarz \Cref{lem:MilnorSchwarz}, there exists a generating set $X_0$ for $H$ such that $\Cay{H}{X_0}$ is equivariantly quasi-isometric to $\C \ell^\orth/\langle T\rangle$. Then the set $X\coloneq X_0\cap \{T^k\}_{k\neq 0}$ generates $\Stab{G}{\ell}$ and $\Cay{\Stab{G}{\ell}}{X}$ is equivariantly quasi-isometric to $\C \ell^\orth/\langle T\rangle$, and therefore to $\C \ell^\orth$ since $\langle T\rangle$ acts on it with uniformly bounded orbits. We therefore once and for all replace every $G$-translate of $\C \ell^\orth$ by the corresponding translate of $\Cay{\Stab{G}{\ell}}{X}$ in the HHG structure.
    
    By \Cref{lem:qmorph_from_acylindrical}, there exists a non-trivial homogeneous quasimorphism $Q\colon H\to \R$ which is $X_0$-coarsely Lipschitz; in turn, if we denote by $p_\ell$ the projection $\Stab{G}{\ell}\to H$, the composition $\psi\coloneq \Stab{G}{\ell}\xrightarrow{p_\ell} H \xrightarrow{Q} \R $ is a non-trivial homogeneous quasimorphism which is $X$-coarsely-Lipschitz. 
    
    Now let $\sigma\colon \Stab{G}{\ell}\to \Z/2\Z$ be the action on the ideal endpoints of $\C \ell$, with kernel $K$. Let $\phi\colon \Stab{G}{\ell}\to \R$ be the Busemann quasicocycle for the action on $\C \ell$, as in \Cref{ex:busemann}; in the case where $K\lneq \Stab{G}{\ell}$, the latter depends on the choice of some element $i\in \Stab{G}{\ell}-K$. By \Cref{cor:quasiline_from_Busemann_qc}, up to equivariantly replacing $\C \ell$ and its translates in $(G,\frakS)$, we can assume that it coincides with the quasiline produced by \Cref{lem:quasiline_from_qc}, and therefore that $\phi\colon \C\ell\to \R$ is a quasi-isometry.
    
    For every $\lambda\in \R$ consider the linear combination
    \[\theta_\lambda\coloneq\phi+\lambda\psi\colon \Stab{G}{\ell}\to \R\]
    which is again a $\sigma$-quasicocycle; notice that $\theta_\lambda(T)=\phi(T)\neq 0$, so $\theta_\lambda$ is unbounded. \Cref{lem:quasiline_from_qc} then produces a generating set $S_\lambda$ for $\Stab{G}{\ell}$ such that $L_\lambda\coloneq \Cay{\Stab{G}{\ell}}{S_\lambda}$ is a quasiline and $\theta_\lambda\colon L_\lambda\to \R$ is a quasi-isometry. 
\end{construction}

\noindent The following is the key technical step of \Cref{thm:uncountably_many_CM}:
\begin{prop}\label{prop:new_hhg_structures}
     In the setting of \Cref{notation:tweakable_hhg} and \Cref{construction:newline}, for any $\lambda\in \R$ there exists a HHG structure for $(G,\frakS_\lambda)$ such that:
     \begin{itemize}
         \item The underlying metric space is $Z$, and the domain set is $\frakS_\lambda=\frakS$, with the same relations;
         \item For every $W\not\in G\cdot \ell$, $\C W$ is unchanged;
         \item For every $g\in G$, $\C{(g\ell)}=g L_\lambda$.
     \end{itemize}
     Furthermore, $\langle T\rangle$ acts loxodromically on $L_\lambda$ and with uniformly bounded orbits on $\C W$ whenever $W\in \frakS_\lambda$ is orthogonal to $\ell$. 
\end{prop}

\begin{proof}
We shall prove one axiom from \Cref{defn:HHS} and \Cref{defn:HHG} at a time, introducing all relevant data as soon as we shall need them. Each axiom will be proven with some choice of the HHS constant, and we shall eventually set the HHS constant for $(G,\frakS_\lambda)$ as the maximum of these values. To avoid confusion, when we want to distinguish between the original and the new projections, we denote the copy of $\ell$ for the structure depending on $\lambda$ by $\ell_\lambda$: for instance, $\pi_{\ell_\lambda}$ will denote the new projection to $L_\lambda$, while $\pi_\ell$ will denote the original projection to $\C \ell$. For two subsets $A,B$ of a metric space, we will use the notation $A\sim B$ to mean that the Hausdorff distance between $A$ and $B$ is bounded by a uniform constant, possibly depending on the original structure $(Z,\frakS)$ and on $\lambda$.

\medskip\textbf{(Projections.)} If $W\not\in G\cdot \ell$, let $\pi_W\colon Z\to \C W$ be the original projection. Next, fix an identification of $\Stab{G}{\ell}$ with $P_\ell$ given by cobounded product regions, so that for all $g\in G$ the coset $g\Stab{G}{\ell}$ is identified with the subspace $gP_\ell=P_{g\ell}$; then set $\pi_{g\ell_\lambda}\colon Z\to gL_\lambda$ as the composition of the gate map $\gate_{g\ell}\colon Z\to g\Stab{G}{\ell}$ and the identity map $g\Stab{G}{L}\to gL_\lambda=g\Cay{\Stab{G}{\ell}}{S_\lambda}$. We stress that $\gate_{g\ell}$ is defined using the original HHG structure, and as such does not depend on $\lambda$. The projection $\pi_{\ell_\lambda}$ does not depend on the choice of the coset representative $g$, since both $\gate_\ell$ and the projection to $L_\lambda$ are $\Stab{G}{\ell}$-equivariant; furthermore, it is a coarsely Lipschitz coarse map as so are its factors. Finally, by construction $g\pi_{\ell_\lambda}(x)=\pi_{g\ell_\lambda}(gx)$ for all $x\in Z$ and $g\in G$, so the new projections are uniformly coarsely Lipschitz coarse maps. 
        \\
        For later purposes, we point out that $\pi_{\ell_\lambda}$ only depends on the original projections on $\C \ell$ and $\C \ell^\orth$:
        \begin{claim}\label{claim:what_determines_pilambda}
            For every $A\subseteq G$, $\diam \pi_{\ell_\lambda}(A)$ is bounded in terms of $\max\{\diam \pi_{\ell}(A), \diam \pi_{\ell^\orth}(A)\}$.
        \end{claim}
        \begin{claimproof}[Proof of \Cref{claim:what_determines_pilambda}] There is nothing to prove if either $\pi_{\ell}(A)$ or $\pi_{\ell^\orth}(A)$ are unbounded, so suppose otherwise. Taking the gate to $P_{\ell}$ does not change the coordinates in $\C \ell$ and $\C \ell^\orth$, so $B\coloneq \gate_\ell(A)$ has bounded projections to $\C \ell$ and $\C \ell^\orth$. Now, by construction $\pi_{\ell_\lambda}(A)$ is the projection of $B$ to $L_\lambda$; since $\theta_\lambda=\phi+\lambda\psi\colon L_\lambda\to \R$ is a quasi-isometry, it is enough to show that both $\phi$ and $\psi$ are bounded on $B$, when we identify the latter with a subset of $\Stab{G}{\ell}$.  Now we analyse one quasicocycle at a time:
            \begin{itemize}
                \item Recall that $\phi$ is the Busemann quasicocycle associated to the action on $\C\ell$, as in \Cref{ex:busemann}, which depends on the choice of $i\in \Stab{G}{\ell}$ and has the property that $\phi(ki)=\phi(k)$ for all $k\in K$. Now let $x\in  B$, and let $x'$ be either $x$ if $x\in K$, or $xi^{-1}$ if $x\not\in K$, so that $\phi(x')=\phi(x)$ while  $\pi_{\ell}(x')\sim \pi_\ell(x)$ belongs to a bounded subset. If we fix $x_0\in B$, for every other $x$ the difference $x'(x_0')^{-1}$ preserves a bounded subset, and therefore acts with uniformly bounded orbits on $\C\ell$ as it preserves orientation. Hence $\phi (x'(x_0')^{-1})=0$ for all $x\in  B$, meaning that $\phi(B)\sim \phi(x_0')$ is bounded.
                \item Furthermore, $\C\ell^\orth =\Cay{\Stab{G}{\ell}}{X}$ and $\psi$ is $X$-coarsely-Lipschitz, hence $\psi(B)$ is bounded as so is $\pi_{\ell^\orth}(B)$.\qedhere
            \end{itemize}
        \end{claimproof}

        \medskip\textbf{(Nesting, finite complexity, orthogonality and containers.)} The relations on $\frakS$ are the same as before, so these axioms already hold. Moreover, since the only domains whose coordinate spaces have been changed are $\nest$-minimal, whenever $U\propnest V$ we can define $\rho^U_V$ as before. 
		
        \medskip\textbf{(Transversality.)} If $V\trans W$ and none is in the $G$-orbit of $\ell$, set $\rho^V_W$ as before. If instead, without loss of generality, $V=\ell$, we set $\rho^W_{\ell_\lambda}=\pi_{\ell_\lambda}(P_W)$, where $P_W$ is the product region for $W$. To see that the latter is a uniformly bounded subset of $L_\lambda$, first notice that $W\trans \ell$ implies $W\trans \ell^\orth$ by \Cref{notation:tweakable_hhg}.\eqref{need:orth_compl}. Furthermore, $\pi_\ell(P_W)\sim \rho^W_\ell$ and $\pi_{\ell^\orth}(P_W)\sim \rho^W_{\ell^\orth}$, by definition of the product region, and therefore $\pi_{\ell_\lambda}(P_W)$ is uniformly bounded by \Cref{claim:what_determines_pilambda}. Notice also that the new relative projections are $G$-equivariant by construction.

        \medskip\textbf{(Consistency.)}
        Let $x\in Z$, and let $V,W$ be two transverse domains. If none is a translate of $\ell$ there is nothing to prove, so up to the $G$-action assume that $V=\ell$. Suppose first that $W\not\in G\cdot \ell$; we shall prove the consistency inequality holds with some constant which is at least $4E$. Indeed, if $\dist_{W}(x,\rho^\ell_W)>4E$ then $\dist_{W}(x,\rho^{\ell^\orth}_W)> E$, where we used \Cref{lem:DHS}, the fact that $W$ is also transverse to $\ell^\orth$ by \Cref{notation:tweakable_hhg}.\eqref{need:orth_compl}, and that $\diam\rho^\ell_W\le E$ by definition of the HHG structure. The original consistency inequalities in $(Z,\frakS)$ then yield that $\max\{\dist_{\ell}(x,\rho^W_\ell),\dist_{\ell^\orth}(x,\rho^W_{\ell^\orth})\}\le E$, so $\dist_{L_\lambda}(x, \rho^W_{\ell_\lambda})=\dist_{L_\lambda}(x, P_W)$ is uniformly bounded by \Cref{claim:what_determines_pilambda}.
        
        Now suppose that $\ell\trans g\ell$ for some $g\in G-\{1\}$. By repeatedly applying \Cref{notation:tweakable_hhg}.\eqref{need:orth_compl} we see that $\ell\trans g\ell^\orth$, $\ell^\orth\trans g\ell$, and $\ell^\orth\trans g\ell^\orth$. Now, by \Cref{claim:what_determines_pilambda} we can assume that $\dist_{g\ell_\lambda}(x,\rho^\ell_{g\ell_\lambda})=\dist_{g\ell_\lambda}(x,P_\ell)$ is large enough that $\max\{\dist_{g\ell}(x,\rho^\ell_{g\ell}),\dist_{g\ell^\orth}(x,\rho^\ell_{g\ell^\orth})\}>4E$. Then the above argument, where $W$ is whichever between $g\ell$ and $g\ell^\orth$ realises the maximum, runs verbatim to show that $\dist_{L_\lambda}(x, \rho^W_{\ell_\lambda})$ is uniformly bounded. If $W=g\ell$ we are done. If instead $W=g\ell^\orth$ then notice that $\rho^{g\ell}_{\ell}\sim \rho^{g\ell^\orth}_\ell$ in $\C \ell$ and $\rho^{g\ell}_{\ell^\orth}\sim \rho^{g\ell^\orth}_{\ell^\orth}$ in $\C \ell^\orth$ by \Cref{lem:DHS}; therefore \[\rho^{g\ell}_{\ell_\lambda}=\pi_{L_\lambda}(P_{g\ell})\sim \pi_{L_\lambda}(P_{g\ell^\orth})= \rho^{g\ell^\orth}_{\ell_\lambda},\] where we used \Cref{claim:what_determines_pilambda} and how product regions project to domains.
        
        Finally, let $U\propnest V$ and assume both $U$ and $V$ have a well-defined relative projection to $W$. Again there is nothing to prove if $W\not\in G\cdot \ell$, so suppose $W=\ell$ up to the $G$-action. Since $\ell$ is $\nest$-minimal, we must have $V\trans  \ell$ and $U\trans  \ell$, so $V\trans  \ell^\orth$ and $U\trans  \ell^\orth$ as well by \Cref{notation:tweakable_hhg}.\eqref{need:orth_compl}. By the original consistency inequalities, $\rho^U_\ell\sim \rho^V_\ell$ and $\rho^U_{\ell^\orth}\sim \rho^V_{\ell^\orth}$; thus
        \[\rho^U_{\ell_\lambda}=\pi_{L_\lambda}(P_U)\sim\pi_{L_\lambda}(P_V) =\rho^V_{\ell_\lambda},\]
        where we again used \Cref{claim:what_determines_pilambda} and how the product regions $P_U$ and $P_V$ project to $\C\ell$ and $\C\ell^\orth$.

        \medskip\textbf{(Bounded geodesic image.)} Let $x,y\in G$ and let $V\propnest W$; since we only replaced minimal domains, we just have to consider the case $V=\ell$. 
        Notice that $W\not\bot\ell^\orth$, since $(\ell^\orth)^\orth=\ell$ by \Cref{notation:tweakable_hhg}.\eqref{need:orth_compl}, and clearly $W\not\nest \ell^\orth$ or it would be orthogonal to $\ell$. Then either $\ell^\orth$ is either transverse to, or properly contained inside, $W$.
        Now let $\gamma\subseteq \C W$ be a geodesic between $\pi_W(x)$ and $\pi_W(y)$. If $\dist_W(\gamma, \rho^\ell_W)>4E$, then $\dist_W(\gamma, \rho^{\ell^\orth}_W)>E$ as well by \Cref{lem:DHS}. Now the original bounded geodesic image axiom tells us that $\dist_\ell(x,y)\le E$; furthermore, by either the bounded geodesic image axiom (if $ \ell^\orth\propnest W$) or the consistency inequalities (if $\ell^\orth\trans W$), we also get that $\dist_{\ell^\orth}(x,y)$ is uniformly bounded. Then $\dist_{L_\lambda}(x,y)$ is uniformly bounded by \Cref{claim:what_determines_pilambda}.
		
        \medskip\textbf{(Partial realization.)} The proof of this axiom will rely on the following claim, which also establishes the ``furthermore" part of the Theorem:
        \begin{claim}\label{claim:action_of_T}
            $\langle T\rangle$ acts loxodromically on $L_\lambda$ and with uniformly bounded orbits on $\C W$ whenever $W\in \frakS_\lambda$ is orthogonal to $\ell$.
        \end{claim}

        \begin{claimproof}
            Since $\theta_\lambda(T)=\phi(T)\neq 0$, $\langle T\rangle$ still acts loxodromically on $L_\lambda$. If $W\orth\ell$ is not in the $G$-orbit of $\ell$, then $\langle T\rangle$ already acted with uniformly bounded orbits on $\C W$ by \Cref{notation:tweakable_hhg}.\eqref{need:dehntwist}, so there is nothing to prove. If instead $\ell\orth g\ell$ for some $g\in G$, by symmetry it is enough to show that $gTg^{-1}$ acts with uniformly bounded orbits on $L_\lambda$. Recall that $gTg^{-1}$ fixes by $\ell$ \Cref{notation:tweakable_hhg}.\eqref{need:dehntwist}, so it belongs to $\Stab{G}{\ell}$; moreover, again by assumption, $g\langle T\rangle g^{-1}$ acts with uniformly bounded orbits on the original $\C\ell$, so it must commute with $T$ which acts loxodromically. Thus $gTg^{-1}$ belongs to the kernel of the $\Stab{G}{\ell}$-action on $L_\lambda$. By \Cref{lem:quasiline_from_qc}.\eqref{item:qc3}, we are therefore left to prove that $\theta_\lambda(gTg^{-1})=0$. This is because $\phi(gTg^{-1})=0$ (here we are using that $\phi$ is the Busemann quasicocycle for the action on $\C \ell$); moreover, $\psi(gTg^{-1})=0$ by \Cref{cor:bounded_orbit_implies_trivial}, since $gTg^{-1}$ fixes the bounded subset $\rho^{g\ell}_{\ell^\orth}\subset \C \ell^\orth$.
        \end{claimproof}

        We now turn to the proof of partial realisation. Let $\mathcal V=\{V_1,\ldots,V_k\}\subset \frakS$ be an orthogonal collection. Up to reordering $\mathcal V$ suppose that $V_i=g_i\ell$ if and only if $i\le T$ for some $T\le k$. Fix $t\le T$; then for all $i\le t$ choose $p_i\in g_iL_\lambda$, and for all $i>t$ choose $p_i\in \C V_i$. We shall prove that, for every $t$, there exist a uniform constant $K_t$ and an element $g\in G$ $K_t$-coarsely realising the given coordinates, meaning that:
        \begin{itemize}
            \item For all $i\le t$, $\dist_{g_iL_\lambda}(g,p_i)\le K_t$;
            \item For all $i> t$, $\dist_{\C V_i}(g,p_i)\le K_t$;
            \item If $W\not \in G\cdot \ell$ is transverse to, or properly contains, some $V_i$, then $\dist_{\C W}(g,\rho^{V_i}_W)\le K_t$;
            \item If $W=g\ell$ for some $g\in G$ and $W$ is transverse to, or properly contains, some $V_i$, then $\dist_{gL_\lambda}(g,\rho^{V_i}_{g\ell_\lambda})\le K_t$.
        \end{itemize}
        Since since there is a uniform bound on the cardinality of orthogonal collections by \Cref{lem:pairwise_orth_are_at_most_E}, there are finitely many possible values of $t$, so setting $K=\max K_{t}$ will give a uniform constant not depending on $t$. This will then imply the partial realisation axiom for $(G,\frakS_\lambda)$, which is the case $t=T$. 

        We shall proceed by induction on $t$, so we first assume $t=0$. By the partial realization axiom for $(Z,\frakS)$, there exists $g\in G$ whose coordinates in $\frakS$ $E$-coarsely realise the tuple $\{(V_i,p_i)\}$, so the only bullet point of the above list we have to check is the last one. Up to the $G$-action by translation, suppose some $V_i$ is transverse to $\ell$, and therefore also to $\ell^\orth$ by \Cref{notation:tweakable_hhg}.\eqref{need:orth_compl}. By definition of the realisation point in $\frakS$, $\max\{\dist_{\ell}(g,\rho^{V_i}_{\ell}),\dist_{\ell^\orth}(g,\rho^{V_i}_{\ell^\orth})\}\le E$; in turn, by \Cref{claim:what_determines_pilambda}, there exists a uniform constant $K_0\ge E$ such that $\dist_{L_\lambda}(g,\rho^{V_i}_{\ell_\lambda})\le K_0$, and we are done.

        Now suppose $t'>0$ and that the conclusion holds with $t'-1$ coordinates chosen from the new quasilines. Up to the $G$-action, suppose that $V_1=\ell$, and choose any $q\in \C \ell$. By induction, the tuple $\{(\ell,q)\}\cup\{(V_i,p_i)\}_{i=2,\ldots,k} $ is $K_{t-1}$-coarsely realised by some $g_0\in G$. By \Cref{claim:action_of_T} the $\langle T\rangle$-action on $L_\lambda$ is $C$-cobounded for some $C$, so we can find a power $m\in \Z$ such that $g\coloneq T^mg_0$ satisfies $\dist_{L_\lambda}(g,p_1)\le C$. We now check that $g\in G$ $K_t$-coarsely realises $\{(V_i,p_i)\}_{i=1,\ldots, k}$, for some uniform constant $K_t\ge 0$.
         \begin{itemize}
            \item By construction $\dist_{L_\lambda}(g,p_1)\le C$; moreover, since $\dist_{L_\lambda}(g_0,p_i)\le K_{t-1}$ for all $2\le i\le t$, and since $\langle T\rangle$ acts with uniformly bounded orbits on $g_i \ell$ by \Cref{claim:action_of_T}, we see that $\dist_{L_\lambda}(g,p_i)$ is uniformly bounded. 
            \item For the same reason as above, $\dist_{\C V_i}(g,p_i)$ is uniformly bounded for all $i>t$.
            \item Suppose $W\not \in G\cdot \ell$ is transverse to, or properly contains, some $V_i$. Since $\langle t\rangle$ fixes every $V_i$ by \Cref{notation:tweakable_hhg}, $T^{-m}W$ is also transverse to, or properly contains, some $V_i$. Thus, by $G$-equivariance of the original (relative) projections, $\dist_{\C W}(g,\rho^{V_i}_W)=\dist_{\C (g^{-1}W)}(g_0,\rho^{V_i}_{g^{-1}W})\le K_{t-1}.$
            \item The same argument, but using that the new projections are $G$-equivariant, shows that, whenever $W=g\ell$ is transverse to, or properly contains, some $V_i$, then $\dist_{gL_\lambda}(g,\rho^{V_i}_{g\ell_\lambda})\le K_{t-1}$.
        \end{itemize}

            \medskip\textbf{(Uniqueness.)} Let $r\ge 0$, and suppose $x,y\in Z$ satisfy $\dist_{W}(x,y)\le r$ for all $W\in \frakS-G\cdot \ell$ and $\dist_{gL_\lambda}(x,y)\le r$ for all $g\in G$. If we prove that $\dist_{\C (g \ell)}(x,y)$ is uniformly bounded (meaning, independently of $g$, $x$, and $y$), then the uniqueness axiom for the original structure will bound $\dist_Z(x,y)$ in terms of $r$; furthermore, up to the $G$-action by translation, it is enough to show that $\dist_{\C (\ell)}(x,y)$ is uniformly bounded (meaning, independently of $x$ and $y$). Let $x'=\gate_\ell(x)$ and $y'=\gate_\ell(y)$, which both lie in $P_\ell=\Stab{G}{\ell}$; notice that, by how the gate is defined, $x'$ (resp. $y'$) projects $E$-close to $x$ (resp. $y$) both in $\C\ell$ and $\C\ell^\orth$. Since $\theta_\lambda\colon\Stab{G}{\ell}\to L_\lambda$ is a quasi-isometry, and since $\pi_{L_\lambda(x)}$ (resp. $\pi_{L_\lambda(x)}$) is the image of $x'$ (resp. $y'$) in $L_\lambda(x)$, the difference 
        \[|\theta_\lambda(x')-\theta_\lambda(y')|=|(\phi(x')-\phi(y'))+\lambda(\psi(x')-\psi(y'))|\]
        is uniformly bounded. Furthermore, since $\psi\colon \Stab{G}{\ell}\to \R$ is $\C\ell^\orth$-coarsely-Lipschitz, we also get that $|\psi(x')-\psi(y')|$ is uniformly bounded; hence $|\phi(x')-\phi(h')|$ is uniformly bounded by the triangular inequality. But then this means that $\dist_{\C\ell}(x',y')$ is uniformly bounded, because in \Cref{construction:newline} we assumed that $\C\ell$ is the quasiline associated to $\phi$ and therefore $\phi\colon \C\ell\to \R$ is a quasi-isometry.

        \medskip\textbf{(Large links.)} We shall prove the passing up axiom \ref{axiom:passing_up}. Fix a constant $D$, greater than the original HHS constant $E$, such that, for any two points $x,y\in Z$, if $\dist_{\C\ell}(x,y)\le E$ and $\dist_{\C\ell^\orth}(x,y)\le 3E$, then $\dist_{L_\lambda}(x,y)\le D$, which exists by \Cref{claim:what_determines_pilambda}. Now fix $t>0$, and let $P(t)$ be the integer appearing in the Passing up axiom for $(Z,\frakS)$. We shall prove the passing up axiom with HHS constant at least $D$ and with at least $Q(t)=2P(t)-2+P(21E)$ domains. More precisely, we shall prove the contrapositive: given $V\in \frakS$, $x,y\in Z$, and $\{U_i\}_{i=1}^Q\subseteq \frakS_V$ satisfying $\dist_{U_i}(x,y)>D$, we shall assume that no domain $W$ inside $V$ properly contains some $U_i$ and is such that $\dist_W(x,y)>t$, and we shall prove that $Q\le Q(t)$. 
        
        By the original passing up axiom, at most $P(t)-1$ of the $U_i$ are not in the $G$-orbit of $\ell$, so all the other $U_i$ are of the form $g_i \ell$ for some $g_i\in G$. For each such $i$, either $\dist_{\C(g_i\ell)}(x,y)> E$ or $\dist_{\C(g_i\ell)^\orth}(x,y)> 3E$, by how we defined $D$. Again, the original passing up axiom implies that at most $P(t)-1$ of the $g_i \ell$ can satisfy $\dist_{\C(g_i\ell)}(x,y)> E$, so there are at least $Q'=Q-2P(t)-2$ domains of the form $U_i=g_i\ell$ satisfying $\dist_{\C(g_i\ell)^\orth}(x,y)> 3E$.

        \begin{figure}[htp]
            \centering
            \includegraphics[width=0.8\linewidth]{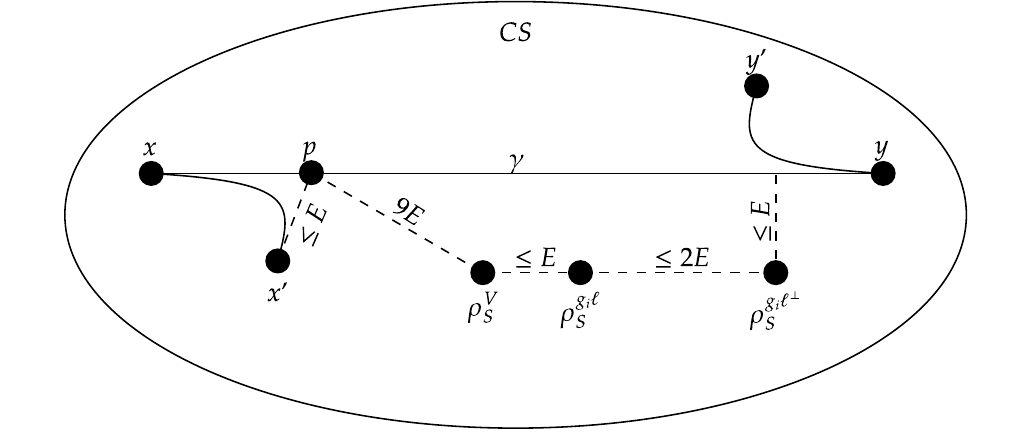}
            \caption{The projections to $\C S$ of the various domains and points appearing in the proof of the Passing up axiom. The geodesic between $x$ and $x'$ is outside the $6E$-neighbourhood of $\rho^V_S$, and in particular it does not pass $E$-close to any $\rho^{g_i\ell^\orth}_S$ in the collection. }
            \label{fig:passingup}
        \end{figure}

        We are left to prove that $Q'\le P(21E)$. Consider the projection of $x$, $y$, and the $Q'$ domains above in the top-level coordinate space $\C S$: see \Cref{fig:passingup} to understand the situation. All $g_i\ell$ are nested in $V$, so $\dist_S(\rho^{g_i\ell}_S, \rho^{V}_S)\le E$ by consistency. Furthermore, $\dist_S(\rho^{g_i\ell}_S, \rho^{g_i\ell^\orth}_S)\le 2E$ by \Cref{lem:DHS}. Now let $\gamma =[\pi_S(x), \pi_S(y)]\subseteq \C S$ be a geodesic; since $\dist_{\C(g_i\ell)^\orth}(x,y)> 3E$, $\gamma$ intersects the $E$-neighbourhood of $\rho^{g_i\ell}_S$ by bounded geodesic image, and therefore the $(6E)$-neighbourhood of $\rho^V_S$ (here we are using that relative projections have diameter at most $E$). Let $p\in \C S$ be the first point on $\gamma$ within distance $9E$ from $\rho^V_S$. If $p=x$ define $x'=x$; else, let $x'\in G$ project $E$-close to $p$ in $\C S$, which exists because $\pi_S$ is $E$-coarsely-onto. Notice that any geodesic in $\C S$ connecting $x$ and $x'$ avoids the $(6E)$-neighbourhood of $\rho^V_S$ (here we are using that triangles are $E$-slim in $\C S$), and therefore $\dist_{\C(g_i\ell)^\orth}(x,x')\le E$ for all the $Q'$ domains by the bounded geodesic image theorem. Now define $y'$ analogously on the side of $y$, so that $\dist_{\C(g_i\ell)^\orth}(y,y')\le E$ for all the $Q'$ domains. Notice that every $W\in \frakS-\{S\}$ containing some $g_i\ell^\orth$ is bounded by \Cref{notation:tweakable_hhg}.\eqref{need:orth_compl}, and therefore has diameter at most $E$ by the bounded domain dichotomy, \Cref{rem:BDD}. Notice also that 
        \[\dist_S(x',y')\le \dist_S(x',\rho^V_S)+\dist_S(y',\rho^V_S)+E\le 21E,\]
        so the two facts combine to give that $\dist_W(x',y')\le 21E$ whenever $W$ properly contains some $g_i\ell^\orth$. However, for all $i=1,\ldots, Q'$, we have
        \[\dist_{\C(g_i\ell)^\orth}(x',y')\ge \dist_{\C(g_i\ell)^\orth}(x,y) -\dist_{\C(g_i\ell)^\orth}(x,x')-\dist_{\C(g_i\ell)^\orth}(y,y')>E.\]
        The original passing up axiom, applied to the collection $\{g_i \ell^\orth\}_{i=1,\ldots, Q'}$ inside $S$, then implies that $Q'\le P(21E)$, as required.
       
         \medskip\textbf{(Equivariance)} The $G$-action on $\frakS$ is the same as before, and in particular it is cofinite. Moreover, as pointed out above, all projections to and between domains are $G$-equivariant, so the above data defines a HHG structure for $G$.
\end{proof}
 \noindent We are finally ready to prove the main technical theorem:
\begin{thm}\label{thm:uncountably_many_CM}
    If $G$ is as in \Cref{notation:tweakable_hhg}, it admits a continuum of equivariant coarse median structures.
\end{thm}
\begin{proof} The proof will follow the same lines as that of \cite[Theorem 7.1]{short_HHG:I}. For every $\lambda\in \R$ let $\mu_\lambda$ be the coarse median structure associated to the HHG structure $(G,\frakS_\lambda)$, which is constructed as in \Cref{ex:coarse_median_from_HHS}. Our goal is to show that, as we let $\lambda$ vary in $\R$, the coarse medians $\mu_\lambda$ are all non-equivalent. Fix $\lambda_1\neq \lambda_2$, and recall from \Cref{construction:newline} that for $i=1,2$ the quasiline $L_{\lambda_i}$ is constructed via \Cref{lem:quasiline_from_qc} from the quasicocycle
    \[\theta_{\lambda_i}\coloneq\phi+\lambda_i\psi\colon \Stab{G}{\ell}\to \R.\]
    Fix $g\in \Stab{G}{\ell}$ with $\psi(g)\neq 0$, which exists because $\psi$ is a non-trivial homogeneous quasimorphism, and consider triples of the form $1,T^k, g^l$, for some integers $k$ and $l$. For $i=1,2$ let \[m_i=\mu_{\lambda_i}(1,T^k, g^l)\in Z.\] We want to show that, by choosing appropriate values of $k$ and $l$, the quantity \[\min\{\dist_{L_{\lambda_1}}(m_1, m_2),\dist_{L_{\lambda_2}}(m_1, m_2)\}\] can be made arbitrarily large. This will imply that $m_1$ and $m_2$ can be made arbitrarily far in $Z$, because the projections from $Z$ to both ${L_{\lambda_1}}$ and ${L_{\lambda_2}}$ are coarsely Lipschitz.

To determine the projection of $m_i$ to the quasiline $L_{\lambda_i}$, we have to determine which of the numbers $\theta_{\lambda_i}(1),\theta_{\lambda_i}(T^k), \theta_{\lambda_i}(g^l)$ lies between the other two, because then the projection of $m_i$ will coarsely coincide with that of the corresponding $x_i\in\{1,T^k, g^l\}$. Since $\theta_{\lambda_1}(g)\neq \theta_{\lambda_2}(g)$, we can choose values of $k$ and $l$ such that $x_1\neq x_2$; thus, without loss of generality, $x_1$ is either $1$ or $T^k$, and in particular $\theta_{\lambda_1}(x_1)=\theta_{\lambda_2}(x_1)=\phi(x_1)$. 

\begin{claim}
    $m_1$ is uniformly close to $x_1$, i.e. their distance is bounded independently of $k$ and $l$.
\end{claim}
\begin{claimproof}
    By the uniqueness axiom for the structure $(G,\frakS_{\lambda_2})$, it is enough to show that $\pi_U(m_1)$ uniformly coarsely coincides with $\pi_U(x_1)$ for every domain $U$ (where projections are with respect to the structure $(G,\frakS_{\lambda_2})$). We already know that $\pi_{L_{\lambda_1}}(m_1)$ uniformly coarsely agrees with $\pi_{L_{\lambda_1}}(x_1)$, so by how the coarse median was constructed in \Cref{ex:coarse_median_from_HHS} it is enough to show that $1$ and $T^k$ project uniformly close to every other domain $U$ (in the structure $(G,\frakS_{\lambda_1})$). If $U\orth \ell$ then this is because $\langle z\rangle$ acts with uniformly bounded orbits on the coordinate space $\C U$, by the ``furthermore" part of \Cref{prop:new_hhg_structures}; if instead $\ell\trans  U$ or $\ell\nest U$ then this follows from e.g. \cite[Proposition 2.17]{HRSS_3manifold}.
\end{claimproof}
We can now prove that $\dist_{L_{\lambda_2}}(m_1, m_2)$ can be made arbitrarily large. Indeed, by the above Claim and by the definition of $x_2$, we can replace $m_1$ by $x_1$ and $m_2$ by $x_2$; then, since $\theta_{\lambda_2}\colon L_{\lambda_2}\to \R$ is a quasi-isometry, it is enough to notice that the difference 
$$|\theta_{\lambda_2}(x_1)-\theta_{\lambda_2}(x_2)|=|\theta_{\lambda_1}(x_1)-\theta_{\lambda_2}(x_2)|$$
can be made arbitrarily large.
\end{proof}

\section{Applications}\label{sec:applications}
This Section is devoted to checking that extended mapping class groups and several right-angled Artin groups fit the framework of \Cref{thm:uncountably_many_CM}. 
\subsection{Extended mapping class groups}
Let $S=S^g_{b,p}$ be obtained from a closed, connected, oriented surface of genus $g$ by removing $b$ open disks and $p$ points; we call $S$ a \emph{oriented surface of finite-type}. The \emph{complexity} of $S$ is $\chi(S)=3g-b-p-3$. The \emph{extended mapping class group} of $S$, denoted $\MCG(S)$, is the group of self-homeomorphisms of $S$ fixing the boundary pointwise, up to isotopy relative to the boundary; notice that we allow orientation-reversing homeomorphisms, while the usual mapping class groups is the index-two subgroup that preserves the orientation. We refer to \cite{FarbMargalit} for further background on mapping class groups of surfaces. We are now ready to prove \Cref{thmintro:mcg} from the introduction:
\mcg*

\begin{proof}
	First, we assume that $S$ has no boundary components, and check that $\MCG(S)$ satisfies the requirements of \Cref{notation:tweakable_hhg}. By results of many authors \cite{MM1, MM2, Behrstock, BKMM, HHS_I}, $\MCG(S)$ is hierarchically hyperbolic; here we shall use the explicit HHG structure from \cite[Theorem 9.8]{converse}, where:
\begin{itemize}
    \item $\frakS$ is the collection of isotopy classes of subsurfaces whose connected components are not pairs of pants. The precise definition of the word ``subsurface" is a bit subtle (see \cite[Section 9.1]{converse}), but we won't need to go into details here. 
    \item Nesting and orthogonality respectively correspond to subsurface inclusion and disjointness, up to isotopy (with some little caveats, again see \cite[Section 9.1]{converse}).
    \item For every subsurface $U\in\frakS$ which is not an annulus, $\C U$ is the \emph{curve graph} of $U$, that is, the simplicial graph whose vertices are isotopy classes of simple closed curve on $U$ and edges connect curves with the minimal intersection number; if $U$ has complexity at least $2$, this is equivalent to connecting two curves if they are disjoint. If $U\in \frakS$ is an annulus, $\C U$ is the \emph{annular curve graph} from \cite[Section 2.4]{MM2}, which is a quasiline on which the Dehn twist along $U$ acts loxodromically. 
\end{itemize}
As proven in \cite[Theorem 9.8]{converse}, $(Z,\frakS)$ has clean containers. Furthermore, the structure has cobounded product regions: this is essentially because the stabiliser of a subsurface $U$ is virtually the product of the mapping class groups of the mapping class groups of the connected components of $U$ and $S-U$. We therefore checked \Cref{notation:tweakable_hhg}.\eqref{need:prodregions}.

Now, if $S$ has genus at least $1$, let $\ell$ be a non-separating curve; if instead $S$ is a sphere let $\ell$ be a curve cutting out a pair of pants. Notice that $\ell$ is $\nest$-minimal and its annular curve graph is a quasiline, so \Cref{notation:tweakable_hhg}.\eqref{need:annular_curve_graph} holds. If $S$ has genus at least $1$, the orthogonal complement $\ell^\orth$ is precisely the subsurface $S-\ell$, which has complexity at least one by the assumption on the complexity of $S$ and therefore is not a pair of pants or an annulus; if instead $S$ is a sphere, $\ell^\orth$ is the connected component of $S-\ell$ which is not a pair of pants, which exists again by the assumption on the complexity of $S$. In both cases, $\ell^\orth$ is an open, connected subsurface of complexity at least $1$, so its curve graph is unbounded by \cite[Theorem 1.1]{MM1}. Moving to the second requirement of \Cref{notation:tweakable_hhg}.\eqref{need:orth_compl}, if $S$ has genus at least one then $S$ is the unique subsurface containing $\ell^\orth$, while in the case of a punctured sphere the only non-maximal subsurface containing $\ell^\orth$ is its closure, which contains $\ell$ as its boundary and therefore has bounded curve graph. Next, the orthogonal complement of $\ell^\orth$ is $\ell$ itself (in the sphere case, we are using that pair of pants do not belong to the index set). Finally, every subsurface which is transverse to $\ell$ is also transverse to $\ell^\orth$, as $\ell$ is the boundary of the closure of $\ell^\orth$. Hence \Cref{notation:tweakable_hhg}.\eqref{need:orth_compl} is proven.

Finally, let $T$ be the (right) Dehn twist around $\ell$, which is normal in $\Stab{\MCG(S)}{\ell}$ because a mapping class $f\in \MCG(S)$ conjugates $\langle T\rangle$ to the subgroup generated by the Dehn Twist on $f(\ell)$. Notice that $T$ acts loxodromically on the annular curve graph for $\ell$, it fixes every subsurface $U$ which is disjoint from $\ell$, and it acts with uniformly bounded orbits on $\C U$ by \cite[Lemma 2.3]{MM2}; this proves  \Cref{notation:tweakable_hhg}.\eqref{need:dehntwist}. Regarding \Cref{notation:tweakable_hhg}.\eqref{need:acylindricity} , by e.g. \cite[Proposition 3.19]{FarbMargalit} the quotient $H=\Stab{G}{\ell}/\langle T\rangle$ is virtually the mapping class group of $\ell^\orth$, which is not virtually cyclic by the assumption on the complexity and acts acylindrically on $\C \ell^\orth$ by e.g. \cite{Bowditch_tight}.

We finally turn to the case where $S=S^b_{g,p}$ with $b> 0$. Again by e.g. \cite[Proposition 3.19]{FarbMargalit}, there exists a central extension
	\[0\to \Z^b\to \MCG(S)\xrightarrow{q} \mathcal M(S')\to 0,\]
	where $S'$ is obtained from $S$ by gluing a once-punctured disk to each boundary component and $\mathcal M(S')$ is a finite-index subgroup of $\MCG(S')$; notice in particular that, by \Cref{rem: fi_sgr_is_HHG}, $\mathcal M(S')$ inherits a continuum of HHG structures $(\mathcal M(S'),\frakS_\lambda)$ whose associated equivariant coarse median structure are all different. By combining  e.g. \cite[Proposition 6.4]{Fujiwara-Kapovich} and \cite[Theorem 3.12]{FournierFacioMangioniSisto_bcohom_HHG2026}, $\MCG(S)$ has a HHG structure whose coordinate projections are either certain quasilines for the centre, or are obtained by composing the $1$-Lipschitz projection $q\colon \MCG(S)\to\mathcal M(S')$ with the coordinate projections for $\mathcal M(S')$. In particular, by \Cref{ex:coarse_median_from_HHS}, for every $x,y,z\in \MCG(S)$ and every HHS structure for $\mathcal M(S')$, the $q$-projection of the coarse median of $x,y,z$ in $\MCG(S)$ is precisely the coarse median of $q(x),q(y),q(z)$ in $\mathcal M(S')$. In turn, since $q$ is $1$-Lipschitz, the continuum of HHG structures for $\mathcal M(S')$ produces a continuum of equivariant coarse median structures for $\MCG(S)$, as required.
\end{proof}



\subsection{Right-angled Coxeter groups} In what follows, $\Gamma$ is always a simplicial graph, and we denote by $V(\Gamma)$ its vertices and by $E(\Gamma)$ its edges. For a vertex $v\in V(\Gamma)$, we denote its link in $\Gamma$ by $\link{\Gamma}{v}$; for a full subgraph $\Lambda\subseteq\Gamma$ we then define $\link{\Gamma}{\Lambda}=\bigcap_{v\in V(\Lambda)}\link{\Gamma}{v}$, and $\Cstar{\Gamma}{\Lambda}$ is the join of $\Lambda$ and $\link{\Gamma}{\Lambda}$. We refer to the survey \cite{dani_racg} for background and references on RACGs, including for the results we shall invoke without proof. We now restate \Cref{thmintro_coxeter} from the introduction:


\coxeter*

\begin{proof} Let $\Sigma_\Gamma$ be the Davis complex for $W_\Gamma$. We first describe the HHG structure $(\Sigma_\Gamma,\frakS)$ on $W_{\Gamma}$ constructed in \cite[Proposition 8.3]{HHS_I}. The latter is stated for a right-angled Artin group acting on the universal cover of the Salvetti complex  but its proof works also for a right-angled Coxeter group acting on its Davis complex, with minor adjustments due to the fact that the vertex set of $\Sigma_\Gamma$ is not in bijection with $W_{\Gamma}$. For simplicity, whenever possible, we shall also reformulate the structure directly in terms of $\Gamma$ and $W_\Gamma$, without referencing the underlying cubical structure, as in e.g. \cite[page 8]{quasiflats}.

Let $\mathcal F$ be the collection of full subgraphs of $\Gamma$ with the following properties:
\begin{itemize}
\item $\Gamma\in \mathcal F$;
\item If $\emptyset\subsetneq \Delta\subseteq \link{\Gamma}{w}$ for some $w\in V(\Gamma)$, and $\Delta$ either contains both $v_+$ and $v_-$ or none of them, then $\Delta\in \mathcal F$.
\end{itemize} 
Notice that $\mathcal F$ contains all links of vertices, since by assumption $v_+$ belongs to some $\link{\Gamma}{w}$ if and only if $v_-$ does. Moreover, whenever $\Lambda,\Delta\in\mathcal F$ have non-empty intersection, then $\Lambda\cap\Delta\in \mathcal F$. Hence $\mathcal F$ satisfies the requirement of \cite[Proposition 8.3]{HHS_I}, and therefore induces the following HHG structure on $W_\Gamma$. 

For a full subgraph $\Lambda\subseteq \Gamma$, let $W_\Lambda$ be the parabolic subgroup generated by the vertices of $\Lambda$. Let $\frakS=\{gW_\Lambda\}_{\sim}$, where $\Lambda\in \mathcal F$, $g\in W_{\Gamma}$, and $\sim$ is the equivalence relation $gW_\Lambda\sim hW_\Lambda$ if $gh^{-1}$ centralises $W_\Lambda$. Declare $[gW_\Lambda]\nest[gW_{\Lambda'}]$ if $\Lambda\subseteq \Lambda'$, and $[gW_\Lambda]\orth[gW_{\Lambda'}]$ if $\Lambda\subseteq\link{\Gamma}{\Lambda'}$. For every $F\in \frakS$, the associated coordinate space is its \emph{factored contact graph}. We defer its definition to \cite[Definition 8.14]{HHS_I}; the only properties we will need are that:
\begin{itemize}
\item If $\Lambda\in \mathcal F$ is neither a vertex nor a proper join, then $\C [W_\Lambda]$ is unbounded. 
\item If $\{a,b\}\in \mathcal F$ are two non-adjacent vertices, then $\C [W_{\{a,b\}}]$ is a quasiline, on which the dihedral subgroup $W_{\{a,b\}}$ acts geometrically.
\end{itemize}
The coarse median structure induced by $(\Sigma_\Gamma,\frakS)$, as in \Cref{ex:coarse_median_from_HHS}, is precisely the cubical one; this follows e.g. from the arguments in \cite[Sections 22 and 23]{CRHK}.

We now check the requirements of \Cref{notation:tweakable_hhg}. Clean containers is immediate: if $\Lambda\subsetneq\Lambda'$ and both belong to $\mathcal F$, the container for $[gW_\Lambda]$ in $[gW_{\Lambda'}]$ is $[gW_{\Delta}]$, where $\Delta=\Lambda'\cap\link{\Gamma}{\Lambda}\in \mathcal F$ (unless $\Delta$ is trivial, and therefore $[gW_\Lambda]$ is orthogonal to no domain in $[gW_{\Lambda'}]$). Furthermore, the product region associated to $[gW_\Lambda]$ is precisely the Davis complex for $W_{\Cstar{\Gamma}{\Lambda}}=\Stab{W_{\Gamma}}{[gW_\Lambda]}$, on which the latter acts cocompactly; hence \Cref{notation:tweakable_hhg}.\eqref{need:prodregions} is proven.

Now let $\link{\Gamma}{v_\pm}\coloneq \link{\Gamma}{v_+}=\link{\Gamma}{v_-}=\link{\Gamma}{\{v_+,v_-\}}$, and let
$\ell=[\{v_+,v_-\}]$; the latter is indeed an element of $\frakS$ since $\link{\Gamma}{v_\pm}$ is non-empty, so there is some $w\in V(\Gamma)$ whose link contains $\{v_+,v_-\}$. Towards proving \Cref{notation:tweakable_hhg}.\eqref{need:annular_curve_graph}, $[W_{\{v_=,v_-\}}]$ is $\nest$-minimal because neither $\{v^+\}$ nor $\{v^-\}$ belong to $\mathcal F$; furthermore $v_+$ and $v_-$ are not adjacent since $v_+\not\in\link{\Gamma}{v_+}=\link{\Gamma}{v_\pm}$, so as mentioned above $\C \ell$ is a quasiline. 

Moving to \Cref{notation:tweakable_hhg}.\eqref{need:orth_compl}, the orthogonal complement of $\ell$ is $\ell^\orth=[W_{\link{\Gamma}{v_\pm}}]$. Notice that $\link{\Gamma}{v}$ is neither a vertex nor a proper join, so $\C \ell^\orth$ is unbounded. Notice also that $\ell^\orth$ is only nested in $[W_{\Gamma}]$, since no other link of a vertex contains $\link{\Gamma}{v_\pm}$; for the same reason $\link{\Gamma}{\link{\Gamma}{v_\pm}}=\{v_+,v_-\}$, so $(\ell^\orth)^\orth=\ell$. In turn, the two properties combined imply that any $W\trans \ell$ must also be transverse to $\ell^\orth$. Indeed, $W\not\nest \ell^\orth$ or it would be orthogonal to $\ell$, $\ell^\orth\not\propnest W$ or $W$ would coincide with $[W_{\Gamma}]$, and $W\not\bot \ell^\orth$ or it would lie inside $(\ell^\orth)^\orth=\ell$. This proves \Cref{notation:tweakable_hhg}.\eqref{need:orth_compl}.

Next, let $T=v_+v_-\in W_{\Gamma}$, which is normal in $\Stab{W_{\Gamma}}{\ell}=\langle v_+,v_-\rangle\times W_{\link{\Gamma}{v_\pm}}$. Notice that $T$ acts loxodromically on $\C\ell$, it fixes all domains nested inside $W_{\link{\Gamma}{v_\pm}}$ by how we defined the index set, and acts trivially on the associated coordinate spaces, so \Cref{notation:tweakable_hhg}.\eqref{need:dehntwist} holds. Finally, $\Stab{W_{\Gamma}}{\ell}/\langle T\rangle$ is an index-two overgroup of $W_{\link{\Gamma}{v_\pm}}$, which is not virtually cyclic because $\link{\Gamma}{v_\pm}$ is neither a simplex nor the suspension of a (possibly empty) simplex; moreover, it acts acylindrically on $\C\ell^\orth/\langle T\rangle=\C\ell^\orth$ because $(W_{\link{\Gamma}{v_\pm}},\frakS_{\ell^\orth})$ is itself a HHG, and therefore we can apply \Cref{lem:acyl_action_on_CS}. This proves \Cref{notation:tweakable_hhg}.\eqref{need:acylindricity}, and concludes the proof of \Cref{thmintro_coxeter}.
\end{proof}

\noindent The coarse median structures from \Cref{thmintro_coxeter} are all equivalent to the standard cubical one when restricted to virtually Abelian parabolic subgroups:
\begin{lem}\label{cor:new_median_is_the_same_on_DT_RACG} Let $H$ be a product of Dihedral parabolic subgroups, and for all $\lambda\in \R$ let $\mu_\lambda$ be the coarse median induced by $(\Sigma_\Gamma,\frakS_\lambda)$. Then $H$ is $\mu_\lambda$-quasiconvex, and the restriction of $\mu_\lambda$ to $H$ is equivalent to the standard cubical coarse median.
\end{lem}

\begin{proof} Let $\{v_i,w_i\}_{i=1,\ldots, k}$ be a collection of pairs of disjoint vertices spanning a join in $\Gamma$, let $D_i=W_{\{v_i,w_i\}}$, $\mathcal U=\{[D_i]\}\subseteq \frakS$, and let $H=\prod_{i=1}^k D_i$. Our goal is to prove that 
\[ \sup_{a,b,c\in H}\dist_{W_{\Gamma}}(\mu(a,b,c), \mu_\lambda(a,b,c))<+\infty,\]
where $\mu=\mu_0$ is the standard cubical coarse median, induced by the original HHG structure $(\Sigma_\Gamma,\frakS)$. By the uniqueness axiom in $(\Sigma_\Gamma,\frakS_\lambda)$, it will be enough to show that $m\coloneq \mu(a,b,c)$ and $m_\lambda\coloneq \mu_\lambda(a,b,c)$ project uniformly close to each $W\in\frakS_\lambda$.

Notice that, in both HHG structures, $H$ acts geometrically on the product region for $\mathcal U$, defined as in \Cref{defn:prodreg}. This is because the only domains in $\frakS_\lambda$ (resp. $\frakS$) where the projection of $H$ is not uniformly bounded are the elements of $\mathcal U$, and each $D_i$ acts geometrically on the associated quasiline and with uniformly bounded orbits on the others (if some $U_i$ is a translate of $\ell$, we are using that $T$ acts loxodromically on $L_\lambda$ and trivially on its translates, by the ``furthermore" part of \Cref{prop:new_hhg_structures}). Notice also that $P_{\mathcal U}$ is median-quasiconvex, by how the coarse median is constructed in \Cref{ex:coarse_median_from_HHS}. Thus, up to modifying both coarse median structures by a uniformly bounded amount, we can assume that $m,m_\lambda\in H$. 

Hence we are left to show that $m$ and $m_\lambda$ project uniformly close to all elements of $\mathcal U$. If $[D_i]\not \in W_\Gamma\cdot \ell$, then the projection $\pi_{[D_i]}$ is the same as in the original structure, and therefore $\pi_{[D_i]}(m)=\pi_{[D_i]}(m_\lambda)$. Suppose instead that $[D_i]=g \ell$ for some $g\in W_\Gamma$, and notice that we can choose $g=1$ as $\ell$ is the equivalence class of a standard Dihedral subgroup. Notice that $H\le \Stab{W_\Gamma}{\ell}$, and that the quasimorphism $\psi\colon \Stab{W_\Gamma}{\ell}\to \R$ from \Cref{construction:newline} is trivial on $H$ by \Cref{cor:bounded_orbit_implies_trivial}, so $\theta_0|_H\equiv \theta_\lambda|_H=\phi_H$. In particular, if without loss of generality $\phi(a)\le \phi(b)\le\phi(c)$, then the projection of $m$ to $\C \ell$ (resp. $m_\lambda$ to $L_\lambda$) uniformly coarsely coincides with the projection of $b$. Finally, since $m\in H$, we also have that $\pi_{\C\ell^\orth}(m)$ and $\pi_{\C\ell^\orth}(b)$ coarsely coincide; hence \Cref{claim:what_determines_pilambda} implies that $\pi_{L_\lambda}(m)$ coarsely coincides with $\pi_{L_\lambda}(b)$, which in turn coarsely coincides with $\pi_{L_\lambda}(m_\lambda)$.
\end{proof}

\subsection{Right-angled Artin groups}
In the same spirit as \Cref{thmintro_coxeter}, we can exhibit a family of right-angled \emph{Artin} groups admitting uncountably many coarse median structures. The combinatorial conditions translate those from \Cref{thmintro_coxeter} to the Artin setting:

\artin*

\begin{proof}
    The HHG structure for $A_\Gamma$ comes from the action on the universal cover of the Salvetti complex $S_\Gamma$. Let $\mathcal F$ be the collection of full subgraphs of $\Gamma$ containing $\Gamma$ and all non-empty subgraphs of links of vertices, which is clearly closed under non-empty intersections and therefore yields a HHG structure by \cite[Proposition 8.3]{HHS_I}. The index set $\frakS$ is again the collection $\{gA_\Lambda\}/\sim$, where $g\in A_\Gamma$, $\Lambda\in \mathcal F$, and both $\mathcal F$ and the equivalence relation $\sim$ are defined as in the Coxeter case. The factored contact graph of $[A_\Lambda]$ satisfies the following properties:
    \begin{itemize}
        \item If $\Lambda$ has at least two vertices and is not a proper join then $\C [A_\Lambda]$ is unbounded.
        \item If $\Lambda$ is a single vertex, then $\C [A_\Lambda]$ is a line, on which $A_\Lambda\cong \Z$ acts geometrically;
    \end{itemize}
    The proof of \Cref{thmintro_coxeter} now runs verbatim, with the following minor adjustments. The domain $[A_{\{v\}}]$ is $\nest$-minimal because $\{v\}$ is minimal under inclusion among subgraphs of $\Gamma$. Moreover, we set $T=v$, and notice that $A_{\link{\Gamma}{v}}=\Stab{A_\Gamma}{v}$ acts non-elementarily on the unbounded hyperbolic space $\C [A_{\link{\Gamma}{v}}]$ since ${\link{\Gamma}{v}}$ is not a single vertex nor a proper join.
\end{proof}

\noindent Finally, the proof of \Cref{cor:new_median_is_the_same_on_DT_RACG} runs verbatim in our current setting, to prove that the new coarse medians all agree when restricted to Abelian parabolic subgroups:
\begin{lem}\label{cor:new_median_is_the_same_on_DT_RAAG}
    Let $\Delta\subseteq \Gamma$ be a simplex, and let $H=A_\Delta$. For all $\lambda\in \R$ let $\mu_\lambda$ be the coarse median induced by $(S_\Gamma,\frakS_\lambda)$. Then $H$ is $\mu_\lambda$-quasiconvex, and the restriction of $\mu_\lambda$ to $H$ is equivalent to the standard cubical coarse median.
\end{lem}

\bibliography{biblio}
\bibliographystyle{alpha}
\end{document}